\documentclass{amsart}
\usepackage{fullpage}
\usepackage{amsmath}
\usepackage{amsfonts}
\usepackage{mathrsfs}
\usepackage{latexsym,epsfig}
\usepackage{mathabx}
\usepackage{amsmath, amscd, amsthm}
\usepackage{color}
\usepackage{xypic}
\newcommand{\RR}[0]{\mathbb{R}}
\newcommand{\HH}[0]{\mathbb{H}}
\newcommand{\HHH}[0]{\mathcal{H}}
\newcommand{\MM}[0]{\mathcal{M}}
\newcommand{\AM}[0]{\mathcal{AM}}

\newcommand{\MCG}[0]{\mathcal{MCG}}

\newcommand{\TT}[0]{\mathcal{T}}

\newcommand{\PP}[0]{\mathcal{P}}
\newcommand{\CC}[0]{\mathcal{C}}
\newcommand{\OO}[0]{\mathcal{O}}

\newcommand{\ZZ}[0]{\mathbb{Z}}
\newcommand{\NN}[0]{\mathbb{N}}

\newcommand{\Ext}[0]{\text{Ext}}
\newcommand{\GG}[0]{\mathcal{G}}
\newcommand{\Fix}[0]{\mathrm{Fix}}

\usepackage{hyperref}

\newtheorem{theorem}{Theorem}[section]
\newtheorem{lemma} [theorem] {Lemma}
\newtheorem{proposition} [theorem] {Proposition}
\newtheorem{corollary} [theorem] {Corollary}

 \newtheorem{remark}[theorem]{Remark}
 \newtheorem{definition}[theorem]{Definition}
 \newtheorem{notation}[theorem]{Notation}

\begin{document}

\title{Elliptic actions on Teichm\"uller space}
\author{Matthew Gentry Durham}
\address{Department of Mathematics, University of Michigan, 3079 East Hall, 530 Church Street, Ann Arbor, MI 48109}
\email{durhamma(at)umich.edu}
\maketitle
\begin{abstract}

Let $S$ be an oriented surface of finite type, $\MCG(S)$ its mapping class group, and $\TT(S)$ its Teichm\"uller space with the Teichm\"uller metric.  Let $H \leq \MCG(S)$ be a finite subgroup and consider the subset of $\TT(S)$ fixed by $H$, $\Fix(H) \subset \TT(S)$.  For any $R>0$, we prove that the set of points whose $H$-orbits have diameter bounded by $R$, $\Fix_R^T(H)$, lives in a bounded neighborhood of $\Fix(H)$.  As an application, we show that the orbit of any point $X \in \TT(S)$ under the action of a finite order mapping class has a fixed coarse barycenter.  By contrast, we show that $\Fix^T_R(H)$ need not be quasiconvex with an explicit family of examples.

\end{abstract}

\section{Introduction}

Let $S$ be a surface of finite topological type, $\MCG(S) = \mathrm{Homeo}^+(S)/\mathrm{Homeo}_0(S)$ its mapping class group, and $\TT(S)$ its Teichm\"uller space, the space of isotopy classes of marked hyperbolic metrics on $S$, which we consider with both the Teichm\"uller $(\TT(S), d_T)$ and Weil-Petersson $(\TT(S), d_{WP})$ metrics.\\

The Nielsen Realization Problem asks whether a finite subgroup $H \leq \MCG(S)$ of the mapping class group of a surface $S$ can be realized as a subgroup $\widetilde{H} \leq \mathrm{Homeo}^+(S)$ which acts by isometries on some metric $\sigma \in \TT(S)$ on $S$.  Kerckhoff \cite{Ker83} proved that the problem in $\TT(S)$ always has a solution by showing that the length functions of curves are convex along Thurston earthquake paths, a result later mirrored for Weil-Petersson geodesics by Wolpert \cite{Wol87}.  There are several other solutions to this problem \cite{Gab92,CJ94,Tro96,BBFS09,HOP12}, none of which is easy.\\

Kerckhoff's main theorem in \cite{Ker83} was the following equivalent formulation:

\begin{theorem}[Theorem 4 in \cite{Ker83}]
Every finite subgroup $H \leq \MCG(S)$ fixes a point in $\TT(S)$.
\end{theorem}

A number of facts follow immediately from Kerckhoff's theorem.  Let $X \in \Fix(H) \subset \TT(S)$ be fixed by $H$.  The quotient $X/H = \OO$ is a hyperbolic 2-orbifold and any hyperbolic structure on $\OO$ lifts to $S$, giving an embedding $i:\TT(\OO) \hookrightarrow \TT(S)$ that is an isometry onto its image in the Teichm\"uller metric.  Since $\Fix(H) = i(\TT(\OO))$, $\Fix(H) \subset \TT(S)$ is a convex submanifold of $(\TT(S), d_T)$.\\

In this paper, we investigate the structure of the set of points in $\TT(S)$ which are moved a bounded Teichm\"uller distance $R>0$ by the action of $H$, the $R$-\emph{almost fixed points}:
$$\Fix^T_R(H) = \{X \in \TT(S)| \mathrm{diam}_T(H \cdot X) < R\}$$

These almost fixed point sets can be viewed as the level sets of the diameter map, $\mathrm{diam}_T: \TT(S) \rightarrow \RR$, given by $X \mapsto \mathrm{diam}_T(X)$.  From this perspective, $\Fix(H) = \mathrm{diam}_T^{-1}(0)$ and $\Fix_R^T(H) = \mathrm{diam}_T^{-1}([0,R))$.\\

In a negatively curved space, the level sets of the diameter map would be convex regular neighborhoods of the set of fixed points.  However, Masur \cite{Mas75} showed that the Teichm\"uller metric is not negatively curved and Minsky \cite{Min96} later showed that this assumption fails profoundly: in the thin parts of $\TT(S)$, the Teichm\"uller metric is quasiisometric to a sup metric on a product space (See Theorem \ref{r:product} below).\\

The results we obtain in this paper contrast the topological constraints coming from covering theory and the geometric flexibility coming from these product regions.  The Main Theorem \ref{r:main} of this paper proves that almost fixed points are uniformly close to fixed points:

\begin{theorem}[Almost fixed points are close to fixed points]\label{r:thm1}
For any $R>0$, there is a constant $R'$ depending only on $R$ and $S$ such that the following holds.  Let $H \leq \MCG(S)$ be a finite subgroup and $\Fix(H)\subset \TT(S)$ its fixed point set.  Then
\[\Fix^T_R(H) \subset \mathcal{N}_{R'}^T(\Fix(H))\]
where $\mathcal{N}_{R'}^T(\Fix(H))$ is the $R'$-neighborhood of $\Fix(H)$.
\end{theorem}

In a $\mathrm{CAT}(0)$ space, a \emph{barycenter} for a bounded set $E$ with radius $R$ is the unique point $b\in E$ around which a ball of radius $R$ contains $E$, $E \subset B_R(b)$.  A \emph{coarse barycenter} for a set $E$ is any point $x \in E$ invariant under the symmetries of $E$ such that $E \subset B_{K\cdot \textrm{diam}(E)+C}(x)$, where $K,C>0$ are uniform constants and $\mathrm{diam}(E)$ is the diameter of $E$.  Note that a coarse barycenter is a barycenter when $K=\frac{1}{2}$ and $C=0$.\\

Using work of Tao \cite{Tao13}, we also prove that orbits of finite order elements of $\MCG(S)$ have coarse barycenters in $(\TT(S),d_T)$:

\begin{theorem} [Coarse barycenters for $(\TT(S), d_T)$] \label{r:thm2}
There are $K, C>0$ such that for any $\sigma \in \TT(S)$ and any finite order $f \in \MCG(S)$, there is a fixed point $X \in Fix(\langle f \rangle)$ such that 
\[d_T(\sigma, X) < K \cdot d_{\TT(S)}(\sigma, f\cdot \sigma) + C\] 
\end{theorem}

Both of Theorems \ref{r:thm1} and \ref{r:thm2} depend crucially on the fact that $\Fix(H)$ comes from a topological covering map, namely that the subsurfaces involved in the geometric considerations in $\TT(S)$ are all lifts of suborbifolds of $\OO$.\\

We say that a subset $Z \subset X$ of a metric space is \emph{quasiconvex} if there is an $L>0$ such that whenever $x,y \in Z$ and $\mathcal{G}_{x,y}$ is a geodesic between them, then $\mathcal{G}_{x,y} \subset \mathcal{N}_L(Z)$.\\

Recall that $\Fix(H) \subset \TT(S)$ is convex in both the Teichm\"uller and Weil-Petersson metrics.  In contrast with Theorems \ref{r:thm1} and \ref{r:thm2}, the following theorem shows that relaxing the condition of being fixed to being almost-fixed dramatically changes convexity properties:

\begin{theorem}[Nonquasiconvexity of $\Fix^T_R(H)$]\label{r:thm3}
There exist a constant $R>0$, a surface $S$, and a finite subgroup $H \leq \MCG(S)$ such that $\Fix^T_R(H)$ is not quasiconvex. 
\end{theorem}

The counterexamples built in Theorem \ref{r:thm3} are based on work of Rafi \cite{Raf14}.  See the discussion after the proof of Theorem \ref{r:thm3} (Theorem \ref{r:non qc} below) for how nonquasiconvexity of $\Fix_R^T(H)$ is a more general phenomenon.\\

Many of the tools and ideas in this paper are motivated by ideas from geometric group theory and the theory surrounding the study of $\MCG(S)$.  Quasiconvexity is a central notion in the theory of Gromov hyperbolic groups and is well-suited to this strong notion of negative curvature.  Given the product structure on the thin parts, quasiconvexity, and thus convexity, in the Teichm\"uller metric are sensitive properties.  The only known convex subsets of $(\TT(S), d_T)$ are its (unique) geodesics, special isometrically embedded copies of $\HH^2$ called Teichm\"uller disks, and the fixed point sets which are at the center of this paper.  As for quasiconvex subsets, the only known additional examples are bounded diameter subsets \cite{LR11}, the aforementioned product regions themselves, orbits of convex cocompact subgroups of $\MCG(S)$ \cite{FM02}, and certain subsets of metrics on pleated surfaces which fill the convex hull of a hyperbolic 3-manifold homeomorphic to $S \times \RR$ \cite{Min93}.  Theorem \ref{r:thm3} (and its generalizations) suggest that it may be difficult to naturally enlarge $\Fix(H)$ to an $H$-invariant quasiconvex subset of $\TT(S)$.\\

We now give a brief sketch of the proof of the Main Theorem \ref{r:main}, whose proof is contained in Section \ref{r:afp section}.\\

Let $R>0$ and suppose $\sigma \in \Fix^T_R(H)$.  Since $d_{WP} \prec d_T$ \cite{Lin74}, there is some $R'>0$ depending only on $R$ and $S$ such that $\sigma \in \Fix^{WP}_{R'}(H)$.  The Weil-Petersson completion of $\TT(S)$, denoted $\widebar{\TT(S)}$ is a complete CAT$(0)$ space to which the action of $H$ extends.  By a basic lemma from CAT$(0)$ geometry \cite[Proposition II.2.7]{BH99}, there exists a barycenter $\widebar{X} \in \widebar{\Fix(H)} \subset \widebar{\TT(S)}$ of $H \cdot \sigma$ in the completion of $\Fix(H)$, with $d_{WP}(\sigma, \widebar{X}) \leq \mathrm{diam}_{WP}(H\cdot \sigma)$.  Using basic properties of the completion, one can find another fixed point $X' \in \Fix(H) \subset \TT(S)$ close to $\widebar{X}$ in $d_{WP}$.  Using the theorems of Brock, Masur-Minsky, Wolpert, Rafi, and the author, the fact that $X'$ has bounded Weil-Petersson distance to $\sigma$ is used to produce another fixed point $X'' \in \Fix(H)$ whose Teichm\"uller distance is coarsely determined only by Dehn (half-)twists around a uniformly bounded number of curves.  Work of Tao \cite{Tao13} implies that these curves must be $H$-symmetric (see Subsection \ref{r:orb Teich} for a definition) and the amount of twisting to be done around the orbit of the different curves is coarsely equal.  Using the hierarchy machinery developed in \cite{MM00} and the quasiisometry model for $(\TT(S), d_T)$ built in \cite{Dur13}, this fact can then be used to construct a sequence of fixed points which, one by one, reduces the distance to $\sigma$ by performing the twisting on each orbit simultaneously.  The result is a fixed point $X \in Fix(H)$ whose distance to $\sigma$ is coarsely determined only by $S$ and $R$.

\subsection{Acknowledgements}

This paper was part of the author's doctoral thesis at the University of Illinois at Chicago.  The author would like to thank his advisor, Daniel Groves, for his guidance, patience, and unwavering support.  He would also like to thank David Dumas, Hao Liang, Howard Masur for interesting conversations and their encouragement.  Special thanks to Dick Canary for carefully reading this paper.

\section{Preliminaries}

For the Teichm\"uller metric, see the books of Hubbard \cite{Hub} and Papadopoulos \cite{Pap07}; see also the survey of Masur \cite{Mas10}.  For the Weil-Petersson metric, see Ahlfors \cite{Ahl61}; see also the survey of Wolpert \cite{Wol07}. 

\subsection{Conventions and notation}

Throughout this paper, let $S = S_{g,n}$ denote an oriented surface of finite complexity, $\xi(S) = 3g-3 + n>0$, with genus $g$ and $n$ punctures.\\

Our methods and calculations are frequently coarse and we introduce some notation for ease of the exposition.  Given two quantities $A,B$, we write $A \prec B$ if there are constants $K,C>0$ depending only on the topology of $S$ such that $A \leq K\cdot B + C$.  If $A \prec B$ and $B \prec A$, then we write $A \asymp B$.\\

Similarly, given a constant $R>0$, we write $A \prec_R B$ if there are constants $K'$ and $C'$ depending only on $R$ and the topology of $S$ such that $A \leq K' \cdot B + C'$, and the same for $A \succ_R B$ and $A \asymp_R B$.\\

If we have $X,Y,$ and $Z$ such that $X \asymp_R Y$ and $Y \asymp_R Z$ (or $\prec_R$, $\asymp_R$), then we also have $X \asymp_R Z$, where the constants are worse for the latter coarse inequality.  As long as we only make such estimates a uniformly bounded number of times depending only on $R$ and $S$, the associated constants will still be uniform in $R$ and $S$.\\

When we write $A= A(B,C)>0$, we mean that $A$ is a positive constant depending only on the objects $B$ and $C$.

\subsection{The thin part and Minsky's product regions}

In \cite{Mas75}, Masur proved that $(\TT(S),d_T)$ is not negatively curved by exhibiting distinct geodesic rays with a common basepoint which remain a bounded distance apart for all time.  In \cite{Min96}[Theorem 6.1], Minsky expanded on Masur's insight, proving that the thin regions of $(\TT(S), d_T)$, where at least one curve is sufficiently short, are quasiisometric to product spaces with a sup metric.  Minsky's Product Regions (Theorem \ref{r:product} below) is arguably the deepest statement about the geometry of the Teichm\"uller metric and much of our coarse geometric approach hinges on it.\\

Let $\gamma = \gamma_1, \dots, \gamma_n$ be a simplex in $\CC(S)$, and let $Thin_{\epsilon}(S, \gamma) = \{\sigma \in \TT(S) \big| l_{\sigma}(\gamma_i) \leq \epsilon\}$, where $l_{\sigma}(\gamma_i)$ is the length of $\gamma_i$ in $\sigma$, for each $i$.  Let 
\[\TT_{\gamma} = \TT(S\setminus \gamma) \times \prod_{\gamma_i \in \gamma} \HH_{\gamma_i}\]
be endowed with the sup metric, where $S\setminus \gamma$ is a disjoint union of punctured surfaces and each $\HH_{\gamma_i}$ is a \emph{horodisk}, that is, a copy of the upper half-plane endowed with the hyperbolic metric.

\begin{theorem}[Product regions; Theorem 6.1 in \cite{Min96}] \label{r:product}
There is an $\epsilon>0$ sufficiently small so that the Fenchel-Nielsen coordinates on $\TT(S)$ give rise to a natural homeomorphism $\Pi: \TT(S) \rightarrow \TT_{\gamma}$, whose restriction to $Thin_{\epsilon}(S, \gamma)$ distorts distances by a bounded additive amount.
\end{theorem}

\subsection{Curve complexes} \label{r:curve}

The heart of the combinatorial approach to studying $\TT(S)$ is the \emph{complex of curves} of $S$, denoted $\CC(S)$, a simplicial complex whose vertices are isotopy classes of simple closed curves on $S$ and adjacency is determined by disjointness.  In the case that $S$ is a once-punctured torus or a four-holed sphere, adjacency is determined by minimal intersection.  In the case that $Y_{\alpha}$ is an annulus in $S$ with core $\alpha$, $\CC(Y_{\alpha}) = \CC(\alpha)$ is the simplicial complex with vertices consisting of paths between the two boundary components of the metric compactification of $\widetilde{Y}_{\alpha}$, the cover of $S$ corresponding to $Y_{\alpha}$, up to homotopy relative to fixing the endpoints on the boundary; two paths are connected by an edge if they have disjoint interiors.\\

We are only interested in the 1-skeleton of $\CC(S)$ with the path metric.  With this metric, we have the foundational theorem of Masur-Minsky \cite{MM99}:

\begin{theorem} \label{r:hyp}
$\CC(S)$ is a Gromov hyperbolic space.
\end{theorem}

See \cite{MM99}, \cite{MM00}, and Schleimer's notes \cite{Schleim} for basics on curve complexes.

\subsection{Pants, markings, and subsurface projections} \label{r:mark}

This subsection briefly introduces two of the fundamental players in the geometric-combinatorial approach of Masur-Minsky, Brock, Rafi, and the author, namely the curve complexes discussed in Subsection \ref{r:curve} and marking complexes.  See \cite{MM00}, \cite{Br03}, \cite{Raf07}, and \cite{Dur13} for the original treatments.\\

A \emph{pair of pants} on $S$ is a maximal simplex in $\CC(S)$, whose complement in $S$ is a disjoint collection of three-holed spheres.  The \emph{pants complex}, denoted $\PP(S)$, is a simplicial complex whose vertices are pairs of pants and two pairs of pants $P_1, P_2$ are connected by an edge if there are two curves $\alpha \in P_1, \beta \in P_2$ such that $P_1\setminus \alpha = P_2 \setminus \beta$, with $\alpha$ intersecting $\beta$ minimally.\\

We frequently use the following insight of Brock \cite{Br03}:
\begin{theorem}
The pants complex $\PP(S)$ is $\MCG(S)$-equivariantly quasiisometric to Teichm\"uller space with the Weil-Petersson metric, $(\TT(S),d_{WP})$.
\end{theorem}

In \cite{MM00}, Masur-Minsky introduce a quasiisometry model for $\MCG(S)$ called the \emph{marking complex}, denoted $\MM(S)$.  A \emph{marking} $\mu \in \MM(S)$ on $S$ is collection of \emph{transverse pairs} $(\alpha, t_{\alpha})$ where the $\alpha$ form a pants decomposition called the \emph{base of} $\mu$, which we denote by $\mathrm{base}(\mu)$, and each $t_{\alpha}$ is a simplex in the annular complex $\CC(\alpha)$ (see \cite{MM00}[Section 2.4]), called the set of \emph{transversals}.  In addition, we assume that our markings are \emph{clean}, which means that the only base curve that each transversal intersects is its paired base curve.\\

Two such (clean) markings are connected by an edge in $\MM(S)$ if they differ by a Dehn twist or half twist around a base curve, called a \emph{twist move}, or a \emph{flip move}, in which a base curve and its transverse curve have switched roles $(\alpha, t_{\alpha}) \mapsto (t_{\alpha}, \alpha)$ (along with some other technical, coarsely inconsequential changes to guarantee the resulting marking is clean).  See \cite{Dur13}[Subsection 2.3] for a discussion.\\

We need the following result from \cite{MM00}:

\begin{theorem}
The marking complex $\MM(S)$ with the graph metric is quasiisometric to $\MCG(S)$ with any word metric.
\end{theorem}

We are often interested in comparing two curves, pairs of pants, or markings in a curve complex, pants complex, or marking complex of some subsurface.  We do so via \emph{subsurface projections}, an essential concept in all that follows.\\

Let $\alpha \in \CC(S)$ be any simplex and let $Y \subset S$ be any subsurface that is not a pair of pants.  The \emph{subsurface projection} of $\alpha$ to $Y$, denoted $\pi_Y(\alpha)\subset \CC(Y)$, is obtained by completing the arcs in $\alpha \cap Y$ along the boundary of a regular neighborhood of $\alpha \cap Y$ and $\partial Y$ to curves in $Y$.  In the case that $Y = Y_{\gamma}$ is an annulus with core $\gamma$, then we let $\pi_{Y_{\gamma}}(\alpha) = \pi_{\gamma}(\alpha)$ be the set of lifts of $\gamma$ to the annular cover $\widetilde{Y}_{\gamma}$ of $S$ which connect the two boundaries of the compactification of $\widetilde{Y}_{\gamma}$.  We remark that in both cases $\pi_Y(\alpha) \subset \CC(Y)$ is a simplex, unless $\alpha \cap Y = \emptyset$ and then $\pi_{Y}(\alpha) = \emptyset$.\\

In the case of a pants decomposition or a marking $\mu \in \MM(S)$, we set $\pi_Y(\mu) = \pi_Y(\mathrm{base}(\mu))$.  If $Y = Y_{\alpha}$ is an annulus with core $\alpha \in \mathrm{base}(\mu)$ and transversal $t_{\alpha}$, then $\pi_{\alpha}(\mu) = t_{\alpha}$.  See \cite{MM00}[Section 2] for more details.\\

When measuring the distance between the projection of two curves or markings to a subsurface, we typically write $d_Y(\pi_Y(\mu_1), \pi_Y(\mu_2)) = d_Y(\mu_1,\mu_2)$.

\subsection{The augmented marking complex} \label{r:AMS section}

In \cite{Dur13}, we built an augmentation of the marking complex by using $\MM(S)$ as the thick part of $\TT(S)$ and then adding Groves-Manning combinatorial horoballs \cite{GM08} along Dehn (half-)twist lines to mimic the thin product regions of $\TT(S)$ from Theorem \ref{r:product}.  Our main theorem was:

\begin{theorem} \label{r:aug qi}
The augmented marking complex, $\AM(S)$, is $\MCG(S)$-equivariantly quasiisometric to $\TT(S)$ in the Teichm\"uller metric.
\end{theorem}

One of the key constructions we need is that of a Groves-Manning combinatorial horoball.  We define them in the simple case over $\ZZ$, as that is all we need for our purposes.\\

The \emph{combinatorial horoball over $\ZZ$}, $\mathcal{H}(\ZZ)$, is the 1-complex with vertices $\mathcal{H}(\ZZ) = \ZZ \times (\{0\} \cup \NN)$ and edges as follows:
\begin{itemize}
\item If $x,y \in \ZZ$ and $m\in \{0\} \cup \NN$ such that $0<|x-y| \leq e^m$, then $(x,m)$ and $(y,m)$ are connected by an edge in $\mathcal{H}(\ZZ)$.
\item If $x \in \ZZ$ and $m\in \{0\} \cup \NN$, then $(x,m)$ is connected to $(x,m+1)$ by an edge.
\end{itemize}

A fact we need from \cite{Dur13} is that combinatorial horoballs are quasiisometric to horodisks:

\begin{lemma}\label{r:horoball qi}
The combinatorial horoball over $\ZZ$, $\HHH(\ZZ)$, is quasiisometric to the horodisk $\HH^2_{\geq 1}$.
\end{lemma}

These combinatorial horoballs play the part in $\AM(S)$ of the horodisks appearing in Minsky's Product Regions Theorem \ref{r:product}.\\

We now recall the definition of $\AM(S)$.\\

An \emph{augmented marking} $\tilde{\mu} \in \AM(S)$ is a marking $\mu$ with a collection of nonnegative integers called the \emph{length data} $D_{\alpha}(\tilde{\mu}) \in \mathbb{N}\cup \{0\}$, one for each $\alpha \in \mathrm{base}(\mu)$.  Two augmented markings $\tilde{\mu}_1, \tilde{\mu}_2 \in \AM(S)$, with underlying markings $\mu_1, \mu_2 \in \MM(S)$, are connected by an edge in $\AM(S)$ if they differ by one of the following types of elementary moves, which extend the elementary moves in $\MM(S)$:

\begin{itemize}
\item \emph{Flip moves}: If  $\mu_1, \mu_2 \in \MM(S)$ differ by a flip move on a transverse pairing $(\alpha, t_{\alpha}) \mapsto (t_{\alpha}, \alpha)$, and if $\tilde{\mu}_1, \tilde{\mu}_2$ have the same base curves and length data, with $D_{\alpha}(\tilde{\mu}_1) = D_{\alpha}(\tilde{\mu}_2) = 0$ for each $\alpha \in \mathrm{base}(\tilde{\mu}_1) = \mathrm{base}(\tilde{\mu}_2)$.
\item \emph{Twist moves}: If $\alpha \in \mathrm{base}(\mu_1) = \mathrm{base}(\mu_2)$, $D_{\alpha}(\tilde{\mu}_1) = D_{\alpha}(\tilde{\mu}_2) = k > 0$, and $\tilde{\mu}_1= T^m_{\alpha} \tilde{\mu}_2$ with $0 < m < e^k$, where $T_{\alpha}$ denotes the positive Dehn (half)twist around $\alpha$.
\item \emph{Vertical moves}: If $\mu_1 = \mu_2$ and there is an $\alpha \in \mathrm{base}(\mu_1) = \mathrm{base}(\mu_2)$ such that $D_{\alpha}(\tilde{\mu}_1) = D_{\alpha}(\tilde{\mu}_2) \pm 1$ and $D_{\beta}(\tilde{\mu}_1) = D_{\beta}(\tilde{\mu}_2)$ for all $\beta \in \mathrm{base}(\mu_1) \setminus \alpha = \mathrm{base}(\mu_2) \setminus \alpha$.
\end{itemize}

It should be clear from the definitions that a metrically distorted copy of $\MM(S)$ sits bijectively at the base of $\AM(S)$.\\

These $D_{\alpha}$ coordinates can be used to give a coarse measurement of the length of a curve in any augmented marking, regardless of whether the curve is in its base.   We emphasize that this measurement records whether a curve is short in $\tilde{\mu}$ and, if so, coarsely how short it is.  Given an augmented marking $\tilde{\mu} \in \AM(S)$ and a curve $\alpha \in \CC(S)$.  We define

\begin{equation*} D_{\alpha}(\tilde{\mu}) = \left\{
\begin{array}{lr}
D_{\alpha} & \text{if } \alpha \in \textrm{base}(\tilde{\mu})\\
0 & \text{otherwise}
\end{array}
\right.
\end{equation*}

Since our above definition of $D_{\alpha}$ coincides with the length coordinate for any $\tilde{\mu} \in \AM(S)$ with $\alpha \in \mathrm{base}(\tilde{\mu})$, we use the same notation for both going forward.  For any $\tilde{\mu} \in \AM(S)$, we note that $D_{\alpha}(\tilde{\mu})=0$ for all but finitely many $\alpha \in \CC(S)$.  We also note that these coarse lengths coordinates, as with Fenchel-Nielsen length coordinates, behave nicely with respect to the action of $\MCG(S)$.  In particular, if $\phi \in \MCG(S)$, then
\[D_{\alpha}(\phi \cdot \tilde{\mu}) = D_{\phi \cdot \alpha}(\tilde{\mu})\]

Combinatorial horoballs are the $\AM(S)$-analogues of annular curve complexes, so we also want to compare augmented markings on combinatorial horoballs.  Doing so requires some technical care, as annular curve complexes are only quasiisometric to $\ZZ$.  We recall some notation from [Subsection 4.2, \cite{Dur13}].\\

For each $\alpha \in \CC(S)$, choose an arc $\beta_{\alpha} \in \CC(\alpha)$.  For any other $\gamma \in \CC(\alpha)$, let $\gamma \cdot \beta_{\alpha}$ denote the algebraic intersection number.  The map $\phi_{\beta_{\alpha}}: \CC(\alpha) \rightarrow \ZZ$ given by $\phi_{\beta_{\alpha}}(\gamma) = \gamma \cdot \beta_{\alpha}$ is a $(1,2)$-quasiisometry independent of the choice of $\beta_{\alpha}$ which records the twisting of $\gamma$ around $\alpha$ relative to $\beta_{\alpha}$.  The idea is that any two arcs in $\CC(\alpha)$ differ by some number of twists around $\beta_{\alpha}$ up to a small bounded additive error.\\

Let $\widehat{\HHH}_{\alpha} = \HHH(\ZZ)$ be the combinatorial horoball over $\ZZ$.  We can now define a projection map $\pi_{\widehat{\HHH}_{\alpha}}:\AM(S) \rightarrow \widehat{\HHH}_{\alpha}$.  For any $\tilde{\mu} \in \AM(S)$,
\begin{equation*} \pi_{\widehat{\HHH}_{\alpha}}(\tilde{\mu}) = \left\{
\begin{array}{lr} \left(\phi_{\beta_{\alpha}}(t_{\alpha}), D_{\alpha}\right) & \text{if } (\alpha, t_{\alpha}, D_{\alpha}) \in \tilde{\mu}\\
\left(\phi_{\beta_{\alpha}}(\pi_{\alpha}(\tilde{\mu})),0\right) & \text{otherwise}
\end{array}
\right.
\end{equation*}

We note that any error coming from a choice of $\beta_{\alpha} \in \CC(\alpha)$ is uniformly bounded.\\
   
We also need to understand how to project an augmented marking to an augmented marking on a subsurface.  First, we recall the definition of the projection of a marking to the marking complex of a subsurface.\\

For $\mu \in \MM(S)$ and subsurface $Y \subset S$, we build the \emph{projection of $\mu$ to a marking on $\MM(Y)$}, $\pi_{\MM(Y)}(\mu)$, inductively as follows.  Choose a curve $\alpha_1 \in \pi_Y(\mu)$ and build a pants decomposition on $Y$ by choosing $\alpha_i \in \pi_{Y \setminus \bigcup_{j=1}^{i-1} \alpha_j} (\mu)$.  Using this pants decomposition as its base, build a marking on $Y$ by choosing transverse pairs $(\alpha_i, \pi_{\alpha_i}(\mu))$.  Define $\pi_{\MM(Y)}(\mu) \subset \MM(Y)$ to be the collection of all markings resulting from varying the choices of the $\alpha_i$.\\

By \cite{MM00}[Lemma 2.4] and  \cite{Ber03}[Lemma 6.1], the freedom in this process builds a bounded diameter subset of $\MM(Y)$.  We remark however that if $\partial Y \subset \mathrm{base}(\mu)$, then $\pi_{\MM(Y)}(\mu)$ is a unique point in $\MM(Y)$, since every curve in $\mathrm{base}(\mu)$ either projects to itself in $\CC(Y)$ or has an empty projection.\\

Our definition above of projection to a horoball $\pi_{\widehat{\HHH}_{\alpha}}$ enables us to compare length and twisting components for $\alpha$ for different augmented markings, but if $\alpha \in \mathrm{base}(\tilde{\mu})$, then the transversal data for $\alpha$ is lost.  In order to build a new augmented marking, we need to maintain the transversal data.  We emphasize that the following projection is not properly a map to any graph, just a way of arranging data we need.

\begin{definition}[Marked horoball projection] \label{r:m hor proj}
Let $\alpha \in \CC(S)$ and $\tilde{\mu} \in \AM(S)$.  The \emph{marked projection} of $\tilde{\mu}$ to $\HHH_{\alpha}$, denoted $\widehat{\pi}_{\alpha}(\tilde{\mu})$, is defined by
\begin{equation*}
\widehat{\pi}_{\alpha}(\tilde{\mu}) = \left\{ 
\begin{array}{rl}
(\alpha, t_{\alpha}, D_{\alpha}) & \text{  if  } (\alpha, t_{\alpha}, D_{\alpha}) \in \tilde{\mu}\\
(\alpha, \pi_{\alpha}(\tilde{\mu}), 0) & \text{  if  } \alpha \notin \text{base}(\tilde{\mu})
\end{array}
\right.
\end{equation*} 
\end{definition}

We now define the projection of an augmented marking to an augmented marking on a subsurface.  For any augmented marking $\tilde{\mu} \in \AM(S)$ and nonannular subsurface $Y \subset S$, we define the \emph{projection of $\tilde{\mu}$ to $\AM(Y)$} by setting $\pi_{\MM(Y)}(\mu)$ to be the base marking of $\pi_{\AM(Y)}(\tilde{\mu})$ and, for each $\alpha \in \mathrm{base}(\pi_{\MM(Y)}(\mu))$, setting $D_{\alpha}(\pi_{\AM(Y)}(\tilde{\mu}))$ equal to $D_{\alpha}(\tilde{\mu})$ if $\alpha \subset Y$ and 0 otherwise.  In the case that $Y \subset S$ is an annulus with core curve $\beta$, then $\pi_{\AM(Y)}(\tilde{\mu}) = \widehat{\pi}_{\HHH_{\beta}}(\tilde{\mu})$.

\subsection{Distance formulae for marking complexes}

Distances in $\PP(S)$, $\MM(S)$, and $\AM(S)$ coarsely equal the distances in the objects they quasiisometrically model.  The Masur-Minsky hierarchy machinery \cite{MM00} provides coarse distance estimates for the former two.  At the center of the distance estimate for $\AM(S)$ is Rafi's combinatorial model for the Teichm\"uller metric \cite{Raf07}[Theorem 6.1],  which is an adaptation of the machinery in \cite{MM00} to the setting of $(\TT(S),d_T)$.\\

\begin{theorem}[Rafi's formula; Theorem 6.1 of \cite{Raf07}] \label{r:Rafi} Let $\epsilon>0$ be as in Theorem \ref{r:product}. Let $\sigma_1, \sigma_2 \in \TT(S)$, define $\Lambda$ to be the set of curves short in both $\sigma_1$ and $\sigma_2$, and define $\Lambda_i$ to be the set of curves in $\sigma_i$ and not in $\Lambda$.  Let $\mu_i$ be the shortest marking for $\sigma_i$.  Then
\[d_{\TT}(\sigma_1, \sigma_2) \asymp \sum_Y \left[d_Y(\mu_1,\mu_2)\right]_k + \sum_{\substack{\alpha \notin \Lambda}} \log \left[d_{\alpha}(\mu_1,\mu_2)\right]_k + \max_{\substack{\alpha\in \Lambda}} d_{\HH_{\alpha}}(\sigma_1, \sigma_2) + \max_{\substack{\alpha \in \Lambda_i\\ i=1,2}} \log \frac{1}{l_{\sigma_i}(\alpha)}\]
\end{theorem}

 The following theorem compiles the work of Masur-Minsky \cite{MM00}, Brock \cite{Br03}, Rafi \cite{Raf07}, and the author \cite{Dur13}, in coarse distance estimates for the marking complexes in terms of subsurface projections.  As one can build the $\AM(S)$ from $\MM(S)$ and $\MM(S)$ from $\PP(S)$ by adding additional layers of data, the distance formulae increase in complexity to account for the additional information.\\

\begin{theorem}[Masur-Minsky, Brock, Rafi, D.]\label{r:distances}

There is a $K>0$ such that the following holds.  For any $X_1, X_2 \in \TT(S)$, let $\tilde{\mu}_1, \tilde{\mu}_2 \in \AM(S)$ be their shortest augmented markings,  $\mu_1, \mu_2 \in \MM(S)$ be the unique underlying markings and $P_1, P_2 \in \PP(S)$ be the unique underyling pants decompositions.\\

In \cite[Theorem 6.12]{MM00}, Masur and Minsky develop a coarse distance formula for $\MM(S)$:

\begin{equation}
d_{\MM(S)}(\mu_1, \mu_2) \asymp \sum_{Y \subset S} \left[d_Y(\mu_1, \mu_2)\right]_K + \sum_{\alpha} [d_{\alpha}(\mu_1, \mu_2)]_K\label{r:distances 1}
\end{equation}

An application of  \cite[Theorem 1.1]{Br03} gives:

\begin{equation}
d_{WP}(X_1, X_2) \asymp \sum_{Y \subset S} \left[d_Y(P_1, P_2)\right]_K \label{r:distances 2}
\end{equation}

In \cite{Dur13}, we reformulated \cite[Theorem 6.1]{Raf07} as:

\begin{equation}
d_{T}(X_1, X_2) \asymp \sum_{Y \subset S} \left[d_Y(\tilde{\mu}_1, \tilde{\mu}_2)\right]_K  + \sum_{\alpha}[d_{\HHH_{\alpha}}(\tilde{\mu}_1, \tilde{\mu}_2)]_K\label{r:distances 3}
\end{equation}

In the above, the $Y \subset S$ are nonannular.
\end{theorem}

As the subsurface projections $\pi_Y$ are defined in terms of the projections of the bases of markings (i.e., pants decompositions) to $\CC(Y)$, it follows that the sum appearing in \eqref{r:distances 2} is precisely a proper subsum of \eqref{r:distances 3}.  It follows that Weil-Petersson distance is (coarsely) shorter than Teichm\"uller distance, $d_{WP} \prec d_T$. 

\begin{remark}[$d_{WP} \prec d_T$] \label{r:wp<t}
It is a theorem of Linch \cite{Lin74} that one only needs a multiplicative constant.
\end{remark}

\begin{remark} [Bounded $d_{WP}$ implies a bounded number of annular large links] \label{r:t v wp}
A key observation we use in the proof of the Main Theorem \ref{r:main} is that points that are a bounded $d_{WP}$ distance apart can only have a uniformly bounded number of large projections to horoballs between their respective shortest augmented markings.  This is because a bound on projections to nonannular subsurfaces places a bound on the number of flip moves and thus a bound on the number of base curves which can appear along any augmented hierarchy path.  See Lemma \ref{r:unif bound} below for more details.
\end{remark}

\subsection{Coarse representatives of points in $\TT(S)$} \label{r:interplay}

We frequently pass back and forth between a point in $\TT(S)$ and its coarse representatives in both $\PP(S)$ and $\AM(S)$.  To aid the clarity of the exposition, we recall the definitions of the quasiisometries between $\PP(S)$ and $(\TT(S),d_{WP})$ \cite{Br03} and $\AM(S)$ and $(\TT(S), d_T)$ \cite{Dur13}.\\

We begin with Brock's theorem by recalling a theorem of Bers:

\begin{theorem}[Bers]\label{r:bers}
There is a constant $L>0$ depending only on the topology of $S$, such that for any point $X \in \TT(S)$, there is a $P_X \in \PP(S)$ with $l_X(\alpha) < L$ for each $\alpha \in P_X$.
\end{theorem}

For any $X \in \TT(S)$, any $P_X \in \PP(S)$ as in Theorem \ref{r:bers} is called a \emph{Bers pants decomposition}.\\

For any $P \in \PP(S)$, define

\[V_L(P) = \left\{X \in \TT(S) | \mathrm{max}_{\alpha \in P} \left\{l_X(\alpha)\right\} < L \right\}\]

Using the convexity of the length functions $l_X$ along Weil-Petersson geodesics \cite{Wol87} and the augmented Teichm\"uller space, $\widebar{\TT(S)}$, in \cite[Proposition 2.2]{Br03}, Brock proves that $V_L(P)$ is convex and has uniformly bounded diameter independent of $P$, a fact we later prove for the orbifold setting in Proposition \ref{r:VP orb} below:

\begin{proposition}[Proposition 2.2 in \cite{Br03}]\label{r:VP}
There is a $D>0$ depending only on $S$ such that for $L>0$ as above and any $P \in \PP(S)$
\[\mathrm{diam}_{WP}(V_L(P)) < D\]
\end{proposition}

Define a map $\phi: \PP(S) \rightarrow \TT(S)$ by $\phi(P) = X_L(P)$, where $X_L(P) \in V_L(P)$.  The content of \cite[Theorem 1.1]{Br03} is that this map is a quasiisometry.  The difficulty of the proof is showing that the reverse identification is coarsely independent of the choice of $P$.\\

Let $\tilde{\mu} \in \AM(S)$ be any augmented marking.  Recall that $\mathrm{base}(\tilde{\mu}) \in \PP(S)$.  For any $P \in \PP(S)$, there are infinitely many augmented markings $\tilde{\mu} \in \AM(S)$ for which $\mathrm{base}(\tilde{\mu}) = P$.  Indeed, for each curve $\alpha \in P$, there is a horoball's worth of choices one could make for a transversal, $t_{\alpha}$, and length coordinate, $D_{\alpha}$.  Thus it follows from the distance formula Theorem \ref{r:distances} that $\mathrm{diam}_{T}(V_L(P)) = \infty$.  In particular, the identification $\tilde{\mu} \mapsto V_L(\mathrm{base}(\tilde{\mu}))$ is far from a quasiisometry in the Teichm\"uller metric.\\

We now briefly recall the definitions of the quasiisometries $G: \AM(S) \rightarrow \TT(S)$ and $F: \TT(S) \rightarrow \AM(S)$ from \cite[Subsection 3.3]{Dur13}.\\

The map $G: \AM(S) \rightarrow \TT(S)$ is defined in terms of Fenchel-Nielsen coordinates.  For any $\tilde{\mu} \in \AM(S)$ with $\tilde{\mu} = (\mu, D_{\alpha_1}, \dots, D_{\alpha_n})$ with the $D_{\alpha_i}$ as in the definition from Subsection \ref{r:AMS section}, let $\mathrm{base}(\tilde{\mu}) = \{\alpha_1, \dots, \alpha_n\}$ be the pants decomposition for the coordinates of $G(\tilde{\mu})$.\\

Let $\epsilon>0$ be as in Theorem \ref{r:product}.  The length coordinates $l_{\alpha_i}$ are given by any choice of $\frac{\epsilon}{e^{D_{\alpha_i}+1}} \leq l_{\alpha_i} \leq \frac{\epsilon}{e^{D_{\alpha_i}+2}}$.  For each $i$, we can use our choice of length data and the transversal to the $\alpha_i$, $t_{\alpha_i}$, to define a twisting coordinate, $\tau_{\alpha_i}(t_{\alpha_i})$, to be the unique twisting number which takes certain geodesic arcs on the pairs of pants to the geodesic representative of the transversal.\\

We remark that $G(\tilde{\mu}) \in V_L(\mathrm{base}(\tilde{\mu}))$, but the choices involved in constructing $G(\tilde{\mu})$ determine a set of uniformly bounded diameter in $(\TT(S), d_T)$.\\

The map $F:\TT(S) \rightarrow \AM(S)$ comes from building particular augmented markings from specially chosen Fenchel-Nielsen coordinates.\\

Let $\alpha \in \CC(S)$ and take any $\sigma \in \TT(S)$.  We can assign a coarse length $d_{\alpha}: \TT(S) \rightarrow \ZZ_{\geq 0}$ to $\alpha$ via $\sigma$ by

\[d_{\alpha}(\sigma) = \left\{ \begin{array}{lr}
\max \left\{k \Big| \frac{\epsilon}{e^{k+1}} < \Ext_{\sigma}(\alpha) < \frac{\epsilon}{e^k}\right\} & \text{if } \Ext_{\sigma}(\alpha) < \epsilon\\
0 & \text{if }\Ext_{\sigma}(\alpha) \geq \epsilon
\end{array} \right. \]
 where $Ext_{\sigma}(\alpha)$ is the extremal length of $\alpha$ in $\sigma$.\\
 
Let $\mu_{\sigma}$ be the shortest marking for $\sigma$, that is, $\mathrm{base}(\mu_{\sigma})= \{\alpha_1, \dots, \alpha_n\}$ is the collection of shortest curves on $S$ in $\sigma$ and the transversals to the $\alpha_i$ are chosen to be as short as possible.  We note that $\mathrm{base}(\tilde{\mu}_{\sigma}) \in \PP(S)$ is by definition a Bers pants decomposition for $\sigma$.\\

Define $F: \TT(S) \rightarrow \AM(S)$ by $F(\sigma) = (\mu_{\sigma}, d_{\alpha_1}(\sigma), \dots, d_{\alpha_n}(\sigma))$.  We call $F(\sigma)$ a \emph{shortest augmented marking} for $\sigma$ and denote it by $\tilde{\mu}_{\sigma}$.\\

Thus the function of the $d_{\alpha_i}$ is to assign length coordinates to the augmented marking $\tilde{\mu}_{\sigma}$; that is, for each $i$, $D_{\alpha_i}(F(\sigma)) = d_{\alpha_i}(\sigma)$.\\

\begin{remark}[Short curves are base curves] \label{r:short base}
Let $\epsilon>0$ be as in Theorem \ref{r:product} and suppose $X \in \TT(S)$ is such that $l_{X}(\alpha) < \epsilon$ for some $\alpha \in \CC(S)$.  It follows from the constuction that $\alpha \in \mathrm{base}(\tilde{\mu}_{X})$, where $\tilde{\mu}_{X} = F(X)$ is a shortest augmented marking for $X$.  That is, short curves are base curves.
\end{remark}

\begin{remark}[Coarse naturality of $F$]\label{r:coarsely natural}
It is clear from the construction that $F$ is coarsely natural with respect to the action of $\MCG(S)$.  More precisely, there is an $M_1>0$ depending only on $S$ such that if $h \in \MCG(S)$ and $X \in \TT(S)$, then $d_{\AM(S)}(h\cdot \tilde{\mu}_{X}, \tilde{\mu}_{h\cdot X}) < M_1$.  This $M_1$ is precisely the diameter of the set of possible choices for $F(X) \in \AM(S)$.
\end{remark}

\subsection{Hierarchies and augmented hierarchy paths} \label{r:hier sub}

In this subsection, we collect two technical lemmata regarding the Masur-Minsky hierarchy machinery.  For the foundational material on hierarchies, see \cite{MM00}; for good technical overviews, see \cite{Min03}, \cite{Ber03}, and \cite{Tao13}; for the construction of augmented hierarchy paths, see \cite{Dur13}.\\

A \emph{hierarchy} $H$ is a collection of geodesics in various curve complexes $g_Y \subset \CC(Y)$, where the $Y \subseteq S$ are subsurfaces.  Attached to any hierarchy is a pair of markings $\mu_1, \mu_2 \in \MM(S)$ and a \emph{base geodesic} $g_H \subset \CC(S)$ whose endpoints are $\mathrm{base}(\tilde{\mu}_1), \mathrm{base}(\tilde{\mu}_2) \subset \CC(S)$.  In \cite{MM00}, Masur-Minsky show how to piece together the geodesics $g_Y \in H$ into paths called \emph{hierarchy paths}, which are uniform quasigeodesics in $\MM(S)$ between $\mu_1$ and $\mu_2$.  In the case where there the markings come with length data, that is $\tilde{\mu}_1, \tilde{\mu}_2 \in \AM(S)$ are augmented markings, we showed in \cite{Dur13} how to build \emph{augmented hierarchy paths} between $\tilde{\mu}_1$ and $\tilde{\mu}_2$, which are also uniform quasigeodesics in $\AM(S)$.\\

As one progresses along an augmented hierarchy path, one makes progress along the geodesics which comprise the underlying hierarchy.  We need to understand which subsurfaces support geodesics in a hierarchy and in what order an augmented hierarchy path traverses these geodesics.\\

There are several technical difficulties to resolving these problems: a given hierarchy may be resolved into any number of augmented hierarchy paths, which need not fellow travel without strong assumptions on $\tilde{\mu_1}$ and $\tilde{\mu_2}$; it is not possible to determine all subsurfaces which support geodesics in a hierarchy, for those depend on, among other things, the choice of base geodesic $g_H \in H$; if two disjoint subsurfaces $Y, Z \subset S$ support geodesics in $H$, then it is possible that two different augmented hierarchy paths based on $H$ can traverse $Y$ and $Z$ in different orders.\\

While there is no easy solution to these issues, there are some useful coarse statements we can make.  The first tells us that subsurface projections coarsely determine in which subsurfaces an augmented hierarchy path spends most of its time:

\begin{lemma}\cite[Lemma 6.2]{MM00} \label{r:large link condition} 
Let $\tilde{\mu}_1, \tilde{\mu}_2 \in \AM(S)$, let $Y \subset S$ a subsurface, and let $K$ be as in Theorem \ref{r:distances}.  If $d_Y(\tilde{\mu}_1, \tilde{\mu}_2)>K$, then $Y$ supports a geodesic $g_Y \in H$ for any hierarchy $H$ between $\tilde{\mu}_1$ and $\tilde{\mu}_2$.
\end{lemma}

Following \cite{MM00}, we call such subsurfaces with large projections \emph{large links}.  Lemma \ref{r:active interval} below gives a coarse description of the subsegments of an augmented hieararchy path as it passes through a large link and says as much as possible about how such subsegments for two different subsurfaces overlap.  We first need some notions, namely \emph{time-order} and \emph{active segment}, the latter of which is related to that of an \emph{active interval} for a subsurface along a Teichm\"uller geodesic \cite{Raf14}.\\

We say a two subsurfaces $X$ and $Y$ \emph{interlock}, and write $X \pitchfork Y$, if $X\cap Y \neq \emptyset$ and neither is properly contained in the other.\\

Let $\Gamma \subset \AM(S)$ be an augmented hierarchy path based on a hierarchy $H$ between $\tilde{\mu}_1,  \tilde{\mu}_2 \in \AM(S)$ and let $Y \subset S$ be any subsurface.  Suppose that $Y, Z \subset S$ are both large links for $\tilde{\mu}_1, \tilde{\mu}_2$.  In \cite{MM00}, Masur-Minsky define a technical notion called \emph{time-order}, which is a partial order on subsurfaces of $S$ which support geodesics in $H$.  Roughly speaking, $Y$ is time-ordered before $Z$ in $H$, written $Y \prec_t Z$, if $\Gamma$ moves through $Y$ before $Z$; importantly, if $Y \pitchfork Z$, then \cite[Lemma 4.18]{MM00} implies that either $Y \prec_t Z$ or $Z \prec_t Y$. \\

For each $\alpha \in \partial Y$, let $\Gamma_{\alpha}\subset \Gamma$ be the (possibly empty) segment of $\Gamma$ where $\alpha$ is in the base of each augmented marking in $\Gamma_{\alpha}$, which is connected by Lemma 5.6 in \cite{Min03}.  The \emph{active segment} of $Y$ along $\Gamma$ is $\Gamma_Y = \bigcap_{\alpha \in \partial Y} \Gamma_{\alpha}$, the smallest segment of $\Gamma$ each of whose augmented markings contains $\partial Y$ in its base.\\

The following lemma follows easily from work in \cite[Sections 4 and 5]{MM00} and \cite[Subsections 4.2 and 4.3]{Dur13}:

\begin{lemma}[Active segments and time order] \label{r:active interval}
Let $\Gamma$ be as above.  Suppose that $X \subset S$ has nonempty active segment.  There is an $M_2>0$ depending only on $S$ such that the following hold:
\begin{enumerate}

\item For any $\tilde{\eta}_1, \tilde{\eta}_2 \in \Gamma$ preceding and following $I_X$, respectively, we have
\[d_X(\tilde{\mu}_1, \tilde{\eta}_1), d_X(\tilde{\mu}_2, \tilde{\eta}_2)< M_2\]

\item If $Y\subset S$ is any subsurface which interlocks $X$ and $Y \prec_t X$, then
\[\mathrm{diam}_X(\Gamma_Y),d_X(\tilde{\mu}_1,\Gamma_Y), \mathrm{diam}_Y(\Gamma_X),d_Y(\tilde{\mu}_2, \Gamma_X)<M_2\] \label{r:active interval 2}

\item Moreover, if $\partial Y \cap \partial X \neq \emptyset$, then $\Gamma_Y$ and $\Gamma_X$ are disjoint.

\end{enumerate}

\end{lemma}

\section{Coarse product regions in $\AM(S)$}\label{r:coarse proj section}

In this section, we analyze subgraphs of $\AM(S)$ which coarsely behave like the Minsky's product regions.  We follow and build on work of Behrstock-Minsky \cite{BM08} for $\MM(S)$.  The main goal of this section is Proposition \ref{r:phi add}, which is crucial for the distance estimates at the end of the proof of the Main Theorem \ref{r:main}.  A reader familiar with the Masur-Minsky machinery can skip this section, referring back to it during the later proofs as needed.\\

In Section 2 of \cite{BM08}, Behrstock-Minsky derive a distance estimate for two points of $\MM(S)$ or $\PP(S)$ whose base markings have curves in common.  We need an analogous statement for $\AM(S)$, which gives a coarse distance estimate for two points in the same Minsky product region (Theorem \ref{r:product}).  We also need to understand how to project to these regions.\\

Let $\Delta \subset \CC(S)$ be a simplex and consider the subset $Q(\Delta) = \{\tilde{\mu}\in \AM(S) | \Delta \subset \mathrm{base}(\tilde{\mu})\}$.  Let $\sigma(\Delta) \subset S \setminus \Delta$ be the collection of complementary subsurfaces which are not pairs of pants.  Subsurface projections give a map

\[\Phi: Q(\Delta) \rightarrow \prod_{Y \in \sigma(\Delta)} \AM(Y)\]

The following is the $\AM(S)$-analogue of [Lemma 2.1, \cite{BM08}] and it appears in \cite{EMR13} without proof, for it follows quickly from the distance formula in Theorem \ref{r:distances}:

\begin{lemma}\label{r:prod qi}
The map $\Phi$ is a $Stab_{\MCG(S)}(\Delta)$-equivariant quasiisometry.
\end{lemma}

There are a couple of immediate corollaries.  First, we have a coarse distance estimate for $Q(\Delta)$:

\begin{corollary}\label{r:distance prod}
For $\tilde{\mu}_1, \tilde{\mu}_2 \in Q(\Delta)$, we have that $d_Y(\tilde{\mu}_1, \tilde{\mu}_2) \asymp 1$ for any $Y \pitchfork \Delta$ and thus

\[d_{\AM(S)}(\tilde{\mu}_1,\tilde{\mu}_2) \asymp \sum_{Y \subset \sigma(\Delta) } \left[d_Y(\tilde{\mu}_1, \tilde{\mu}_2)\right]_K\]

In particular, $Q(\Delta)$ is quasiconvex with constants only depending on $S$.
\end{corollary}

Second, we have a coarse characterization of Minsky's product regions Theorem \ref{r:product}, which is well-known to the experts:

\begin{corollary} \label{r:coarse prod}
Let $\epsilon>0$ be as in Theorem \ref{r:product}.  Let $\Delta \subset \CC(S)$ be a simplex and let $X_1, X_2 \in Thin_{\epsilon, S}(\Delta)$, with $\tilde{\mu}_{X_1}, \tilde{\mu}_{X_2} \in \AM(S)$ their shortest augmented markings.  Then $\tilde{\mu}_{X_1}, \tilde{\mu}_{X_2} \in Q(\Delta)$ and there is a string of $\MCG(S)$-equivariant quasiisometries

\[Thin_{\epsilon, S}(\Delta) \overset{+}{\asymp} \prod_{\alpha \in \Delta} \HH_{\alpha} \times \TT(S\setminus \Delta) \asymp \prod_{\alpha \in \Delta} \HHH_{\alpha} \times \AM(S\setminus \Delta) \asymp Q(\Delta)\]

where $\AM(S\setminus \Delta) = \prod_{Y \subset \sigma(\Delta)} \AM(Y)$ for $Y$ nonannular.
\end{corollary}

The first quasiisometry is that of Minsky's Theorem \ref{r:product}.  The second quasiisometry comes from applying Lemma \ref{r:horoball qi} and Theorem \ref{r:aug qi} to the appropriate components, in the latter case by choosing a shortest augmented marking on each nonhorodisk component.  The third quasiisometry is from Lemma \ref{r:prod qi}.  We remark that, up to quasiisometry, the metric on a product is unimportant.\\

In [Lemma 2.2, \cite{BM08}], Behrstock-Minsky give a coarse estimate from any marking $\MM(S)$ to $Q(\Delta)$.  The following is the analogue for $\AM(S)$ whose proof we omit for it is essentially the same.

\begin{lemma}{Distance to $Q(\Delta)$}\label{r:distance to q}
Let $\tilde{\mu}\in\AM(S)$ and $\Delta \subset \CC(S)$ and simplex.  Then we have \[d_{\AM(S)}(\tilde{\mu}, Q(\Delta)) \asymp \sum_{Y \pitchfork \Delta} [[d_Y(\tilde{\mu}, \Delta)]]_K\]
where if $Y$ is an annulus with core curve $\alpha$, then $d_Y = d_{\HHH_{\alpha}}$.
\end{lemma}

In the proof of the Main Theorem \ref{r:main}, we need to understand how to project any $\tilde{\mu}\in \AM(S)$ to a coarse nearest point in $Q(\Delta)$.  This involves projecting $\tilde{\mu}$ to $\HHH_{\alpha}$ for each $\alpha \in \Delta$ and then completing those projections to an augmented marking by projecting $\tilde{\mu}$ to $\AM(S\setminus \Delta)$.\\

Before we proceed, we need to show that $\pi_{\HHH_{\alpha}}$ and $\pi_{\AM(Y)}$ are Lipschitz.  Since both of these are entirely built out of subsurface projections, Lemmas \ref{r:lip for horo proj} and \ref{r:lip for mark proj} are easy consequences of the following result from \cite{MM00}:

\begin{lemma}[Lipschitz projection; Lemma 2.4 in \cite{MM00}] \label{r:lipschitz}
Let $Z \subset Y \subset S$ be subsurfaces.  For any simplex $\rho \in \CC(Y)$, if $\pi_Z(\rho) \neq \emptyset$, then $\text{diam}_Z(\rho) \leq 3$.  If $Z$ is an annulus, then the bound is 1.
\end{lemma}

\begin{lemma}[Horoball projections are Lipschitz]\label{r:lip for horo proj}
For any nonannular subsurface $Y \subset S$ and $\alpha \in \CC(Y)$, if $\tilde{\mu}_1, \tilde{\mu}_2 \in \AM(Y)$ have $d_{\AM(Y)}(\tilde{\mu}_1, \tilde{\mu}_2) =1$, then $d_{\HHH_{\alpha}}(\tilde{\mu}_1, \tilde{\mu}_2) \asymp 1$.  

\end{lemma}
\begin{lemma}[Marking projections are Lipschitz]\label{r:lip for mark proj}
Let $Z \subset Y \subset S$ be subsurfaces.  For any $\tilde{\mu}_1, \tilde{\mu}_2 \in \AM(Y)$ with $d_{\AM(S)}(\tilde{\mu}_1, \tilde{\mu}_2 )=1$, we have $d_{\AM(Z)}\left(\pi_{\AM(Z)}(\tilde{\mu}_1),\pi_{\AM(Z)}(\tilde{\mu}_1)\right) \asymp 1$.
\end{lemma}

\begin{proof}
The result follows easily from the distance formula in Theorem \ref{r:distances} and Lemmas \ref{r:lipschitz} and \ref{r:lip for horo proj} after the following observation.\\

For any $\alpha \in \mathrm{base}(\tilde{\mu}_1)$ with $D_{\alpha}(\tilde{\mu}_1)>0$, it follows that $\alpha \in \mathrm{base}(\tilde{\mu}_2)$ and $d_{\HHH_{\alpha}}\left(\tilde{\mu}_1,\tilde{\mu}_2\right)\leq 1$.  If in addition $\alpha \in \CC(Z)$, it follows that $\alpha \in \mathrm{base}(\pi_{\AM(Z)}(\tilde{\mu}_1))\cap \mathrm{base}(\pi_{\AM(Z)}(\tilde{\mu}_2)) $.  Thus the transversal and length data of $\alpha$ in $\tilde{\mu}_1$ and $\tilde{\mu}_2$ also descend to $\pi_{\AM(Z)}(\tilde{\mu}_1)$ and $\pi_{\AM(Z)}(\tilde{\mu}_2)$, and $d_{\HHH_{\alpha}}\left(\pi_{\AM(Z)}(\tilde{\mu}_1),\pi_{\AM(Z)}(\tilde{\mu}_2)\right)\asymp 1$.  Thus $\pi_{\AM(Z)}(\tilde{\mu}_1)$ has a short curve if and only if $\pi_{\AM(Z)}(\tilde{\mu}_2)$ has that same short curve.\\

As all other parts of $\pi_{\AM(Z)}(\tilde{\mu}_1)$ and $\pi_{\AM(Z)}(\tilde{\mu}_2)$ are built from horoball and subsurface projections, the conclusion of the lemma follows from Theorem \ref{r:distances} above.

\end{proof}

We can now define the coarse closest point projection to $Q(\Delta)$.  

\begin{definition}[Coarse closest point projection to $Q(\Delta)$]
For any $\tilde{\mu}\in\AM(S)$ and any simplex $\Delta \subset \CC(S)$, define $\phi_{\Delta}: \AM(S) \rightarrow Q(\Delta)$ by
\[\phi_{\alpha}(\tilde{\mu}) = \left(\left(\widehat{\pi}_{\alpha}(\tilde{\mu})\right)_{\alpha \in \Delta}, \pi_{\AM(S\setminus \Delta)}(\tilde{\mu})\right)\]
\end{definition}

It follows immediately from the definition that $d_{\HHH_{\alpha}}(\tilde{\mu},\phi_{\Delta}(\tilde{\mu})) \asymp 1$ for any $\alpha \in \Delta$.\\

We now prove a number of properties of $\phi_{\Delta}$, culminating in Proposition \ref{r:phi add}, which we need for the proof of the Main Theorem \ref{r:main}.  The first lemma states that, for any $\tilde{\mu}\in\AM(S)$, the choices involved in building  $\phi_{\Delta}(\tilde{\mu})$ result in a uniformly bounded set:

\begin{lemma}\label{r:phi bdd}
For any simplex $\Delta \subset \CC(S)$ and $\tilde{\mu} \in \AM(S)$, we have
\[\mathrm{diam}_{\AM(S)}(\phi_{\Delta}(\tilde{\mu})) \asymp 1\]
\end{lemma}

\begin{proof}
This follows from the facts that $\widehat{\pi}_{\alpha}$ and $\pi_{\AM(Y)}$ are uniformly bounded for any $\alpha \in \CC(S)$ and subsurface $Y \subset S$.
\end{proof}

The following lemma proves that $\phi_{\Delta}$ is a coarse closest point projection to $Q(\Delta)$.  More precisely, the lemma shows that $\phi_{\Delta}(\tilde{\mu})$ records the combinatorial data of any augmented marking $\tilde{\mu}$ relative to the complementary components of $S \setminus \Delta$.  In particular, any augmented hierarchy path from $\tilde{\mu}$ to its projection $\phi_{\Delta}(\tilde{\mu})$ moves mainly through subsurfaces which interlock $\Delta$:

\begin{lemma}\label{r:cpj}
For any $\tilde{\mu}\in\AM(S)$ and simplex $\Delta \subset \CC(S)$, we have $d_Y(\tilde{\mu}, \phi_{\Delta}(\tilde{\mu})) \asymp 1$ for any $Y \subset \sigma(\Delta)$.  In particular, 
\[d_{\AM(S)}(\tilde{\mu}, \phi_{\Delta}(\tilde{\mu})) \asymp d_{\AM(S)}(\tilde{\mu},Q(\Delta))\]
\end{lemma}

\begin{proof}
For any $\alpha \in \Delta$, $d_{\HHH_{\alpha}}(\tilde{\mu},\phi_{\Delta}(\tilde{\mu}))$ is bounded by definition of $\widehat{\pi}$.  Similarly, for any nonannular subsurface $Y \subset \sigma(\Delta)$, $d_Y(\tilde{\mu},\phi_{\Delta}(\tilde{\mu}))$ is also bounded by definition of $\pi_{\AM(S \setminus \Delta)}$.  Thus all projections to subsurfaces disjoint from $\Delta$ are bounded and it follows from Theorem \ref{r:distances} and Lemma \ref{r:distance to q} that
\[d_{\AM(S)}(\tilde{\mu},\phi_{\Delta}(\tilde{\mu})) \asymp \sum_{Y\subset S} \left[d_Y(\tilde{\mu},\phi_{\Delta}(\tilde{\mu}))\right]_K = \sum_{Y\pitchfork \Delta} \left[d_Y(\tilde{\mu},\phi_{\Delta}(\tilde{\mu}))\right]_K \asymp d_{\AM(S)}(\tilde{\mu},Q(\Delta))\]
\end{proof}

The next lemma proves that $\phi_{\Delta}$ is Lipschitz:

\begin{lemma}\label{r:phi lipschitz}
For any simplex $\Delta\subset \CC(S)$ and any $\tilde{\mu}_1, \tilde{\mu}_2 \in\AM(S)$ with $d_{\AM(S)}(\tilde{\mu}_1,\tilde{\mu}_2) = 1$, we have $d_{\AM(S)}(\phi_{\Delta}(\tilde{\mu}_1),\phi_{\Delta}(\tilde{\mu}_2)) \asymp 1$.
\end{lemma}

\begin{proof}
Let $\tilde{\mu}_1, \tilde{\mu}_2 \in \AM(S)$ be such that $d_{\AM(S)}(\tilde{\mu}_1,\tilde{\mu}_2) = 1$.  Then
\begin{align*}
d_{\AM(S)}(\phi_{\Delta}(\tilde{\mu}_1)\phi_{\Delta}(\tilde{\mu}_2)) &\asymp \sum_{Y \subset \sigma(\Delta)}\left[d_Y(\phi_{\Delta}(\tilde{\mu}_1),\phi_{\Delta}(\tilde{\mu}_2))\right]_K\\
&= \sum_{\alpha \in \Delta} \left[d_{\HHH_{\alpha}}(\phi_{\Delta}(\tilde{\mu}_1),\phi_{\Delta}(\tilde{\mu}_2))\right]_K + \sum_{Y \subset \left(\sigma(\Delta) \setminus \Delta\right)} \left[d_Y(\phi_{\Delta}(\tilde{\mu}_1),\phi_{\Delta}(\tilde{\mu}_2))\right]_K\\
&\asymp \sum_{\alpha \in \Delta} \left[d_{\HHH_{\alpha}}(\tilde{\mu}_1,\tilde{\mu}_2)\right]_K + d_{\AM(S\setminus \Delta)}(\tilde{\mu}_1,\tilde{\mu}_2)\\
&\asymp 1
\end{align*}
\end{proof}

Finally, the following proposition proves that the composition of closest point projections to disjoint collections of curves coarsely commute.

\begin{proposition}\label{r:phi add}
For any pair of noninterlocking simplices $\Delta_1, \Delta_2 \subset \CC(S)$ and any $\tilde{\mu}\in\AM(S)$, we have
\[d_{\AM(S)}(\phi_{\Delta_1\cup \Delta_2}(\phi_{\Delta_1}(\tilde{\mu})), \phi_{\Delta_1 \cup \Delta_2}(\tilde{\mu}))\asymp 1\] 
\end{proposition}

\begin{proof}
First of all, note that since $\Delta_1$ and $\Delta_2$ do not interlock, equivalently $\mathrm{diam}_{\CC(S)}(\Delta_1 \cup \Delta_2) \leq 1$, it follows from the definitions that $\phi_{\Delta_1 \cup \Delta_2}(\phi_{\Delta_1}(\tilde{\mu})) \in Q(\Delta_1 \cup \Delta_2)$.\\

By definition we have
\[\phi_{\Delta_1 \cup \Delta_2}(\phi_{\Delta_1}(\tilde{\mu})) = \left(\widehat{\pi}_{\beta}\left(\phi_{\Delta_1}(\tilde{\mu})\right)_{\beta \in \Delta_2}, \pi_{\AM(S\setminus \Delta_2)}\left(\big(\widehat{\pi}_{\alpha}(\tilde{\mu})\big)_{\alpha \in \Delta_1 \triangle \Delta_2},\pi_{\AM(S\setminus \Delta_1)}(\tilde{\mu})\right)\right)\]
where $\Delta_1 \triangle \Delta_2 = \Delta_1 \setminus (\Delta_1 \cap \Delta_2)$ is the symmetric difference and  
\[\phi_{\Delta_1 \cup \Delta_2}(\tilde{\mu}) = \left(\big(\widehat{\pi}_{\alpha}(\tilde{\mu})\big)_{\alpha \in \Delta_1},\big(\widehat{\pi}_{\beta}(\tilde{\mu})\big)_{\beta \in \Delta_2}, \pi_{\AM(S\setminus (\Delta_1 \cup \Delta_2)}(\tilde{\mu})\right)\]

Since $\phi_{\Delta_1 \cup \Delta_2}(\phi_{\Delta_1}(\tilde{\mu})), \phi_{\Delta_1 \cup \Delta_2}(\tilde{\mu}) \in Q(\Delta_1 \cup \Delta_2)$, Lemma \ref{r:distance to q} implies that

\[d_{\AM(S)}\left(\phi_{\Delta_1 \cup \Delta_2}(\phi_{\Delta_1}(\tilde{\mu})), \phi_{\Delta_1 \cup \Delta_2}(\tilde{\mu})\right)\asymp \sum_{Y \subset \sigma(\Delta_1 \cup \Delta_2)}\left[d_Y\left(\phi_{\Delta_1 \cup \Delta_2}(\phi_{\Delta_1}(\tilde{\mu})), \phi_{\Delta_1 \cup \Delta_2}(\tilde{\mu})\right)\right]_K\]
Thus we need only to compare projections to the components of $\sigma(\Delta_1 \cup \Delta_2)$.\\

By definition of $\pi_{\AM(Y)}$, if any $\alpha \in \Delta_1$ or $\beta \in \Delta_2$ lies in $\mathrm{base}(\tilde{\mu})$, then the transversal and length data of such a curve descends to both $\phi_{\Delta_1 \cup \Delta_2}(\phi_{\Delta_1}(\tilde{\mu}))$ and $\phi_{\Delta_1 \cup \Delta_2}(\tilde{\mu})$.  On the other hand, if $\alpha \in \Delta_1$ is not in $\mathrm{base}(\tilde{\mu})$, then the length data of $\alpha$ in $\phi_{\Delta_1}(\tilde{\mu})$ is $(\alpha, \pi_{\alpha}(\tilde{\mu}),0)$.   Since $\Delta_1$ and $\Delta_2$ do not interlock, $\alpha \in \mathrm{base}(\phi_{\Delta_1 \cup \Delta_2}\left(\phi_{\Delta_1}(\tilde{\mu}))\right)$ and, by definition of $\widehat{\pi}_{\alpha}$ and $\pi_{\AM(S \setminus \Delta_2)}$, its transversal data is the same as the transversal data of $\alpha$ in $\phi_{\Delta_2}(\tilde{\mu})$, namely $\pi_{\alpha}(\tilde{\mu})$.  It follows in both cases that the distance between the projections of $\phi_{\Delta_1 \cup \Delta_2}(\phi_{\Delta_1}(\tilde{\mu}))$ and $\phi_{\Delta_1 \cup \Delta_2}(\tilde{\mu})$ to any horoball over a curve in $\Delta_1 \cup \Delta_2$ is uniformly bounded.\\

It remains to show that $d_{\AM(S\setminus (\Delta_1 \cup \Delta_2))}\left(\phi_{\Delta_1 \cup \Delta_2}(\phi_{\Delta_1}(\tilde{\mu})),\phi_{\Delta_1 \cup \Delta_2}(\tilde{\mu})\right) \asymp 1$.  This follows from the definition and the fact that marking projections are Lipschitz, Lemma \ref{r:lip for mark proj}.

\end{proof}

\section{Fixed and almost-fixed points} \label{r:fix and afix}

In this section, we collect some of the basic properties of the naturally defined subsets of Teichm\"uller space coming from finite orbifold coverings which are at the center of this paper.  We also describe coarse analogues in the combinatorial setting of $\AM(S)$ and adapt some related work of Tao \cite{Tao13}.

\subsection{Orbifold Teichm\"uller spaces} \label{r:orb Teich}

For the rest of the paper, fix a finite subgroup $H \leq MCG(S)$.  We note that there is a bound on the order of any such finite subgroup $H\leq \MCG(S)$ and the number of its conjugacy classes depending only on $S$ (see \cite{FM12}[Section 7.1]).  As such, it suffices to consider a single such $H$.\\

Fix also a hyperbolic 2-orbifold $\mathcal{O}$ coming from a covering $\pi: X\rightarrow \OO$ with deck transformation group $H$, where $X \in T(S)$ is fixed by $H$, the existence of which is guaranteed by the Nielsen Realization Theorem (see Figure \ref{r:orb example} for an example of such a covering). Recall that $\mathcal{O}$ is essentially a smooth manifold with a finite number of singular neighborhoods.  Because we are assuming that $S$ is oriented and that $H$ preserves that orientation, all such singular neighborhoods are quotients of discs by finite rotations which come from $H$.  As $H$ preserves the metric on $X$ , the hyperbolic metric on $X$ descends to $\OO$ and we may consider its Teichm\"uller space, $\TT(\OO)$.  See \cite{FM02}[Section 7] for a formal definition of $\TT(\OO)$.\\

In this subsection, we explain how to put Fenchel-Nielsen coordinates on $\TT(\OO)$ and how to lift coordinates to $\TT(S)$.\\

\begin{figure}\label{r:orb example}
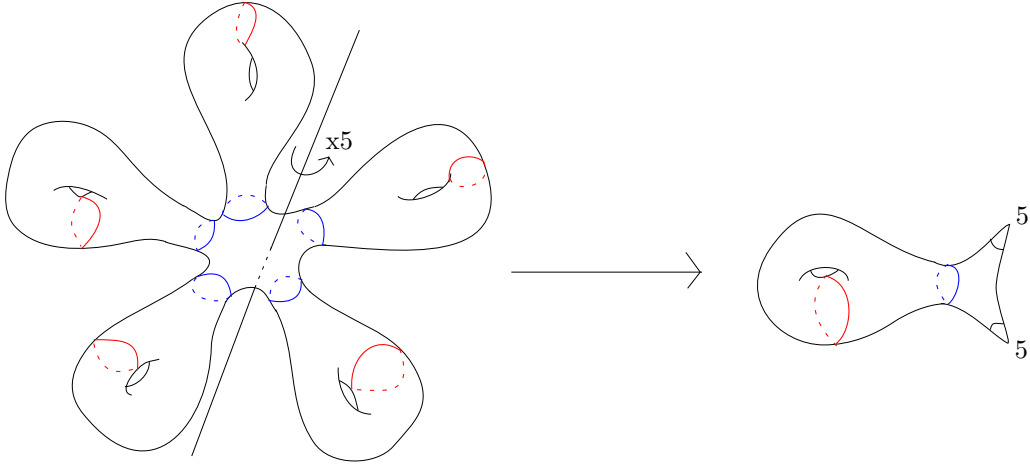
\caption{A simple example of an orbifold cover obtained by quotienting a surface by a finite order mapping class.  A torus with two cone points is obtained as a quotient of a genus five surface rotated around an axis which passes through the center of the surface.  The two cone points come from the two punctures, which are fixed by the rotation.  The \textcolor{red}{red} and \textcolor{blue}{blue} curves are \emph{symmetric} under the rotation and descend to a multicurve on the quotient orbifold.}
\label{fig:g5s3}
\end{figure}

Let $\Delta_i$ be a disjoint collection of small disks around each cone point of $\OO$.  In what follows, we only consider essential, nonperipheral simple closed curves on $\OO\setminus \coprod \Delta_i$.  In particular, we define the \emph{orbifold curve graph of $\OO$}, $\CC(\OO)$, to be the graph whose vertices are homotopy classes of simple closed curves on $\OO$ up to homotopies that do not pass through the $\Delta_i$ and whose edges are given by disjointness.    We note that this is the same condition we impose on curves when $S$ has marked points or punctures.  These assumptions guarantee that any curve $\alpha \in \CC(\OO)$ lifts uniquely to a simplex, $\pi^{-1}(\alpha) \subset \CC(S)$ which is invariant under the action of $H$; we call the lift of any such curve $H$-\emph{symmetric}.  The covering map $\pi:S \rightarrow \OO$ induces a covering relation $\Pi:\CC(\OO)\rightarrow \CC(S)$ given by $\Pi(\beta) = \pi^{-1}(\beta)$.  In \cite[Theorem 8.1]{RS09}, Rafi-Schleimer show that $\Pi$ is a quasiisometric embedding.\\

It is well-known that $\TT(\mathcal{O})$ can be isometrically embedded $i: \TT(\OO) \hookrightarrow \TT(S)$ into $\TT(S)$ with the Teichm\"uller metric (see \cite{RS09} for a brief explanation) as a convex smooth submanifold and that $i(T(\mathcal{O})) = \Fix(H) \subset \TT(S)$ is the fixed set of the action of $H$ on $\TT(S)$.\\

Consider a maximal simplex $A \subset \CC(\OO)$.  The complement $\OO\setminus A$ is a collection of thrice-punctured spheres and spheres with one, two, or three cone points with two, one, or no punctures, respectively (the latter being the degenerate case when $\OO$ is itself a tricornered pillow), which we call an \emph{orbipants decomposition}.  We define the orbipants graph of $\OO$, $\PP(\OO)$, in the same way as $\PP(S)$.  As with a genuine pair of pants, fixing the lengths of the boundary curves in a pair of orbipants uniquely determines a hyperbolic metric thereon, where the order of any cone point plays a fixed role, similar to that of fixing the length of a boundary curve.  By fixing curve lengths and twisting factors when regluing along the curves in $A$, one arrives at Fenchel-Nielsen coordinates for any point $X \in \TT(\OO)$, $(l_{\alpha}(X), t_{\alpha}(X))_{\alpha \in A}$, in nearly the same manner as when $\OO$ is a genuine surface.  We now describe how to induce Fenchel-Nielsen coordinates on $\Fix(H)$ from those on $\TT(\OO)$.\\

The simplex $A \subset \CC(\OO)$ lifts to a simplex $\Pi(A) \subset \CC(S)$.  In order to obtain a pants decomposition on $S$, complete $\Pi(A)$ to a maximal simplex $P \subset \CC(S)$, where $\Pi(A) \subset P$.  The following lemma follows almost immediately from the fact that $i: \TT(\OO)\rightarrow \TT(S)$ is an embedding:

\begin{lemma}[Lifted coordinates] \label{r:lifted coordinates}
Let $X \in \TT(\OO)$ and consider its image $i(X) \in \Fix(H) \subset \TT(S)$.  For any maximal simplex $A \subset \CC(\OO)$ and completion of its lift $\Pi(A) \subset P \subset \CC(S)$ to a maximal simplex, the following hold:
\begin{enumerate}
\item For each $\gamma \in P$, the coordinate pair $(l_{\gamma}(i(X)), t_{\gamma}(i(X)))$ is uniquely determined by the coordinates $(l_{\alpha}(X), t_{\alpha}(X))_{\alpha \in A}$. \label{r:lifted coords 1}
\item For each $\alpha \in A$, there is a number $N_{\alpha} = N_{\alpha}(S)$ such that $l_{i(X)}(\beta) = N_{\alpha} \cdot l_{X}(\alpha)$ for each lift $\beta \in \Pi(\alpha)$. \label{r:lifted coords 2}
\end{enumerate}
Moreover, the number $N_{\alpha}$ is uniformly bounded by a constant depending only on $S$
\end{lemma}
\begin{proof}
(\ref{r:lifted coords 1}) follows from the fact that $i:\TT(\OO) \rightarrow \TT(S)$ is an injection.  (\ref{r:lifted coords 2}) follows from basic covering theory and the fact that $\pi:S \rightarrow \OO$ is a local isometry away from preimages of the cone points.  The constant $N_{\alpha}$ is bounded in terms of $S$ because $|H|$ is and $N_{\alpha} \leq |H|$.
\end{proof}

\begin{notation}[Convention for curves and metrics on $\OO$, and their lifts]
From now on, we adopt a bar notation, $\widebar{\alpha} \in \CC(\OO)$, for curves on $\OO$ and denote their lifts by $\alpha = \Pi(\widebar{\alpha}) \subset \CC(S)$.  Similarly, $\widebar{X} \in \TT(\OO)$ lifts uniquely to $X \in \Fix(H) \subset \TT(S)$.
\end{notation}

We remark that Bers' Theorem \ref{r:bers} holds in the setting of $\TT(\OO)$:

\begin{corollary}\label{r:bers orb}
There is a constant $L'>0$ depending only on $\OO$ so that for any $X \in \TT(\OO)$, there exists $\widebar{P}_X \in \PP(\OO)$ of $\OO$ with $l_{X}(\widebar{\alpha})<L$ for each $\widebar{\alpha} \in \widebar{P}$.
\end{corollary}

Finally, we compile the various notions of a short curve into a single shortness constant:

\begin{definition}[Shortness defined] \label{r:epsilon}
In what follows, fix $\epsilon_0>0$ to be sufficiently small so 
\begin{itemize}
\item Minsky's Product Regions Theorem \ref{r:product} holds,
\item If $\widebar{X} \in \TT(\OO)$ has $l_{\widebar{X}}(\widebar{\alpha})< \epsilon_0$ for some $\widebar{\alpha} \in \CC(\OO)$ and $\widebar{X}$ lifts to $X \in \Fix(H)$, then $l_{X}(\beta) < \epsilon$ for each $\beta \in \Pi(\alpha)$, where $\epsilon>0$ is as in Theorem \ref{r:product},
\item If $L$ is Bers' constant from Theorem \ref{r:bers}, then $\epsilon_0 < L\cdot N_H$, where $|H| <N_H$ depends only on $S$, and if $l_{X}(\gamma) < \epsilon_0$ for some $\gamma \in \CC(S)$, then $l_{X}(\delta) > L$, for any $\delta \in \CC(S)$ with $i(\delta, \gamma) \geq 1$.
\end{itemize}
\end{definition}

Note that such an $\epsilon_0$ depends only on the topology of $S$ by Lemma \ref{r:lifted coordinates} and the Collar Lemma.  When we say that a curve $\alpha$ is \emph{short} for some $\sigma \in \TT(S)$, we mean that $l_{\sigma}(\alpha) < \epsilon_0$.  It follows from Remark \ref{r:short base} that if $l_{\widebar{X}}(\widebar{\alpha}) \leq \epsilon_0$, then $\alpha \in \mathrm{base}(\tilde{\mu}_{X})$, where $\tilde{\mu}_X$ is a shortest augmented marking for $X$.

\subsection{Almost-fixed points, symmetric large links, and Tao's Lemma}

Recall that for any finite $H \leq \MCG(S)$, $\Fix(H) \subset \TT(S)$ is a totally geodesic submanifold, but less is understood if we relax the condition of being fixed by $H$ to being almost-fixed by $H$, that is, having a bounded $H$-orbit.  Our main theorem shows that these almost-fixed points are uniformly close to $\Fix(H)$.  In order to find a fixed point near an almost-fixed point, we need to understand how efficient paths between almost-fixed points and fixed points move through $\TT(S)$.  Using $\AM(S)$, we reduce this to understanding the large links which appear along augmented hierarchy paths between almost-fixed augmented markings and certain almost-fixed augmented markings coming from fixed points in $\TT(S)$.  Using work of Tao \cite{Tao13}, we show that, for a sufficiently large threshold, these large links are $H$-symmetric.  The purpose of this subsection is to make this precise.\\

In \cite{Tao13}, Tao shows that there is an exponential-time algorithm to solve the conjugacy problem for $\MCG(S)$.  The bulk of the work in \cite{Tao13} is proving a number of technical results about hierarchies in the setting of the action of a finite order element of $\MCG(S)$ on $\MM(S)$.  Our first step is an easy extension of some of her results to finite order subgroups acting on $\AM(S)$.\\

Let $H \leq \MCG(S)$ be a finite order subgroup.  For any $R>0$, we define the set of \emph{$R$-almost-fixed points} of $H$ in $\MM(S)$ to be
\[\Fix^{\MM}_R(H) = \{\mu \in \MM(S) | \mathrm{diam}_Y(H\cdot \mu) \leq R, \forall Y \subset S\}\]

For any $R>0$, we define the set of $R$-almost-fixed points of $H$ in $\TT(S)$ in the Teichm\"uller metric to be

\[\Fix^T_R(H) = \{\sigma \in \TT(S) | \mathrm{diam}_T(H \cdot \sigma) \leq R\}\]

Throughout the paper, we work with the coarse version of $\Fix^T_R(H)$, namely $\widetilde{\Fix_R}(H) \subset \AM(S)$ which we define as
\[\widetilde{\Fix_R}(H) = \{\tilde{\mu} \in \AM(S)| \mathrm{diam}_{\AM(S)}(H \cdot \mu) \leq R\}\]

For $\tilde{\mu} \in\widetilde{\Fix_R}(H)$, it follows from Theorem \ref{r:Rafi} that  $d_Y(\mu, h\cdot \mu) \leq K_R$, for each $h \in H, Y \subset S,$ where $K_R$ depends on $R$ and $S$.\\

For the rest of the subsection, fix an arbitrary augmented marking $\widetilde{X} \in \AM(S)$ and an arbitrary almost-fixed augmented marking $\tilde{\mu} \in \widetilde{\Fix}_R(H)$.

\begin{definition} [$H$-symmetric]
We say a subsurface $Y \subset S$ is \emph{symmetric under the action of H} or simply \emph{$H$-symmetric} if each component of $H \cdot Y$ is either $Y$ or disjoint from $Y$.
\end{definition}

Recall from Lemma \ref{r:large link condition} that we call a subsurface $Y \subset S$ a $K$-large link for two augmented markings $\tilde{\mu}_1, \tilde{\mu}_2 \in \AM(S)$ if $d_Y(\tilde{\mu}_1, \tilde{\mu}_2 ) > K$.\\

The following lemma tells us that there is a large link constant $\widetilde{K}$, which depends on $\mathrm{diam}_T(H \cdot \widetilde{X})$, such that any $\widetilde{K}$-large link is $H$-symmetric.  It is an easy adaptation of \cite[Lemma 3.3.4]{Tao13} to our purposes.  We give a proof of the adaptation starting from the basis of her lemma, which is the following lemma in which $H = \langle f\rangle$ for a finite order $f \in \MCG(S)$ and $\MM(S)$ replaces $\AM(S)$.

\begin{lemma} [Symmetric large links; Tao's lemma] \label{r:sym links} 
Let $K>0$ be fixed as above.  There is a $\widehat{K}=\widehat{K}(\widetilde{X},R,S)>0$ such that the following hold:
\begin{enumerate}
\item If $\tilde{\mu} \in \widetilde{\Fix}_R(H)$ and $Y \subset S$ satisfies $d_Y(\tilde{\mu}, \widetilde{X}) >\widehat{K}$, then the orbit $H\cdot Y$ is disjoint, $d_Z(\tilde{\mu},\widetilde{X})>\widehat{K}$ for each component $Z \subset H\cdot Y$, and none of the components of $H \cdot Y$ is time-ordered with respect to any other. \label{r:sym links 1}
\item For any horoball $\HHH_{\alpha}$, if $d_{\HHH_{\alpha}}(\tilde{\mu}, \widetilde{X})>\widehat{K}$, then $\alpha$ is $H$-symmetric. \label{r:sym links 2}
\end{enumerate}
\end{lemma}

\begin{proof}
\cite[Lemma 3.3.4]{Tao13} implies (\ref{r:sym links 1}) for each $f \in H$ and any $Y \subset S$ which is not horoball.  We first extend the result to all of $H$.  It suffices to show that none of the components of $H \cdot Y$ is time-ordered with respect to any other.  Suppose that $Y \subset S$ is a subsurface such that for some $f,g \in H$ we have $f \cdot Y \prec_t g \cdot Y$.  Since $f \cdot Y$ is contained in the orbit of $g \cdot Y$ under the action of $f \cdot g^{-1} \in H$, \cite[Lemma 3.3.4]{Tao13} implies that $f \cdot Y$ and $g \cdot Y$ cannot be time-ordered, which is a contradiction.\\

Let $\HHH_{\alpha}$ be a $\widebar{K}$-large link, where $\widebar{K}>K$ and $K$ is the constant from \cite[Lemma 3.3.4]{Tao13} which depends on $\widetilde{X}$.  If $d_{\alpha}(\tilde{\mu}, \widetilde{X})> K$, then $\alpha$ is $H$-symmetric and we are done.  Otherwise, it must be the case that the $\alpha$-length coordinates of $\tilde{\mu}$ and $\widetilde{X}$ are bounded away from each other, that is $|D_{\alpha}(\tilde{\mu}) - D_{\alpha}(\widetilde{X})|>2R$, for $\widebar{K}$ sufficiently large and $R$ the almost-fixed constant for $\tilde{\mu}$.  If $D_{h\cdot \alpha}(\widetilde{X})=0$ for some $h \in H$, then $D_{h\cdot \alpha}(\tilde{\mu})>R$, and thus $D_{g\cdot \alpha}(\tilde{\mu})>0$ for each $g \in H$ because $\tilde{\mu} \in \widetilde{\Fix}_R(H)$, proving that $\alpha$ is $H$-symmetric.  Similarly, if $D_{h\cdot \alpha}(\widetilde{X})>0$ for each $h \in H$, then we must also have that $\alpha$ is $H$-symmetric.  This completes the proof of (\ref{r:sym links 2}).
\end{proof}

Thus $\widehat{K}$-large links between any augmented marking and an almost-fixed augmented marking partition into $H$-invariant \emph{symmetric families}.

\begin{remark}[Bad domains]\label{r:bad constant}
For the remainder of the paper, fix $\widehat{K}$ as in Lemma \ref{r:sym links}.  In \cite{Tao13}, subsurfaces in $\mathcal{L}_{\widehat{K}}(\widetilde{X},\tilde{\mu})$ were called \emph{bad domains}, though we do not use this terminology here.

\end{remark}

\begin{remark}[Dependence of $\widehat{K}$]
The dependence of $\widehat{K}$ on $\mathrm{diam}_T(\widetilde{X})$ in Lemma \ref{r:sym links} means that $\widehat{K}$ depends only on $R$ and $S$ when $\widetilde{X} \in \Fix_R^T(H)$.  In particular, the constant $R'$ in the Main Theorem \ref{r:main} below is independent of the choice of $R$-almost-fixed point.  Similarly, the constants in the coarse barycenter Theorem \ref{r:bary teich} are independent of the choice of $X \in \TT(S)$.
\end{remark}

While the hierarchical time-ordering is generally not preserved by the action of $\MCG(S)$, the following lemma gives an important exception:

\begin{lemma}\label{r:sym fam time order}
Let $\widetilde{X} \in \AM(S)$, $\tilde{\mu} \in \widetilde{\Fix}_R(H)$.  Suppose $Y, Z \in \mathcal{L}_{\widehat{K}}(\widetilde{X}, \tilde{\mu})$ are $\widehat{K}$-large links with distinct symmetric families and that $Y \pitchfork Z$.  If $Y \prec_t Z$ and $g \cdot Y \pitchfork Z$ for some $g \in H$, then $g \cdot Y \prec_t Z$.  
\end{lemma}

\begin{proof}
Since $g \cdot Y \pitchfork Z$, \cite[Lemma 4.18]{MM00} implies that either $g\cdot Y \prec_t Z$ or $Z \prec_t g\cdot Y$.  In the latter case, transitivity of $\prec_t$ implies $Y \prec_t Z \prec_t g\cdot Y$, a contradiction of Lemma \ref{r:sym links}.
\end{proof}

\begin{remark}
Recall that an $H$-symmetric subsurface $Z$ may have $h \cdot Z = Z$ for each $h \in H$.  If $Y \pitchfork Z$, it is possible that $h\cdot Y \pitchfork Z$ for all $h \in H$.  In this case, Lemmas \ref{r:active interval} and \ref{r:sym fam time order} tell us that the active segment of $Z$ along any augmented hierarchy path either comes entirely before or entirely after the active segments of each subsurface in $H\cdot Y$.
\end{remark}

Another immediate consequence of the finite order of $H$ is that subsurface projections within a symmetric family are all coarsely equal, with constants depending on $\widehat{K}$:

\begin{lemma}[Subsurface projections for symmetric families] \label{r:sym sub proj}
Let $\widetilde{X} \in \AM(S)$ and $\tilde{\mu} \in \widetilde{\Fix}_R(H)$.  If $Y \in \mathcal{L}_{\widehat{K}}(\widetilde{X},\tilde{\mu})$, then for all $h,g \in H$ 
\[d_{h\cdot Y} (\widetilde{X}, \tilde{\mu}) \asymp_{\widehat{K}} d_{g \cdot Y}(\widetilde{X}, \tilde{\mu})\]
where $d_Y = d_{\HHH_{\alpha}}$ if $Y$ is an annulus with core curve $\alpha$.
\end{lemma}

\subsection{Adjusting lengths of short curves for fixed points}

In this subsection, we prove that adjusting the lengths of short curves in a fixed point only results in a bounded change in the Weil-Petersson metric and does not introduce any other short curves, an observation which is crucial for the proof of Proposition \ref{r:reduce short curves} below.  We obtain this as a consequence of a version of Brock's Proposition \ref{r:VP} for our setting.  We do not have all the tools of Masur and Wolpert's work in the setting of $\TT(\OO)$, so we must use the symmetry of the covering action.\\

Before introducing Proposition \ref{r:VP orb} below, we recall some facts about $\TT(S)$ in $d_{WP}$.  In the Weil-Petersson metric, $\TT(S)$ is an incomplete CAT(0) space \cite{Wol87} and its completion, the augmented Teichm\"uller space $\widebar{\TT(S)}$, is obtained as a union of Teichm\"uller spaces of noded surfaces \cite{Mas76}, where disjoint collections of simple closed curves on $S$ have been pinched down to points.  This layers $\widebar{\TT(S)}$ into strata, with the combinatorics of the adjacency of the strata determined by $\CC(S)$.  Importantly, each stratum is WP-geodesically convex \cite{Wol86}.  The incompleteness of $\TT(S)$ in $d_{WP}$ comes from the fact that there are Weil-Petersson geodesic rays which converge to metrics on noded surfaces in finite time.  See \cite{MW02} and \cite{Br05} for more details.\\

We now recall a theorem of Wolpert \cite{Wol05}.  Let $\alpha_1, \dots, \alpha_k \in \CC(S)$ be a collection of disjoint curves.  Let $X \in \TT(S)$ and consider the length sum 

\[l = l_{X}(\alpha_1) + \cdots + l_{X}(\alpha_k)\]

\begin{theorem}[Corollary 21, \cite{Wol05}]\label{r:Wol est}
For any $X \in \TT(S)$, the minimal distance from $X$ to a surface, $Z$, noded along $\alpha_1, \dots, \alpha_k$ is
\[d_{\widebar{WP}}(X,Z) = \sqrt{2 \pi l} + O(l^2)\] 
  
\end{theorem}

Let $A \in \CC(S)$ be any simplex and recall the following definition from Subsection \ref{r:interplay} 

\[V_L({A}) = \{{X} \in \TT(S)| l_{{X}}(\alpha) < L, \forall {\alpha} \in {A}\}\]
where $L$ is the Bers constant from Theorem \ref{r:bers}.\\

For any simplex $A \subset \CC(S)$, let $\TT(S, A) \subset \widebar{\TT(S)}$ be the stratum of marked noded surfaces which are noded along $A$.  Recall that each point in $\TT(S,A)$ is defined by choice of a point in $\TT(Y)$ for each nonpants component $Y \subset S\setminus A$.  Since length functions are convex along Weil-Petersson geodesics \cite{Wol87}, each stratum $\TT(S,A)$ is convex in $d_{WP}$.  We also note that it follows from Wolpert's Theorem \ref{r:Wol est} that $d_{\widebar{WP}}(X, \TT(S,A)) \asymp_L 1$ for any $X \in V_L(A)$.\\

\begin{proposition}\label{r:VP orb}
Let $\widebar{P} \subset \CC(\OO)$ be any orbipants decomposition of $\OO$.  For any $\delta>0$, $V_{\delta}(\widebar{P}) \subset \TT(\OO)$ satisfies
\[\mathrm{diam}_{WP}(V_{\delta}(\widebar{P})) \asymp_{\delta} 1\]

\end{proposition}

\begin{proof}

Consider the lift $P \subset \CC(S)$ of $\widebar{P}$ to $S$.  While $\widebar{P}$ is an orbipants decomposition of $\OO$, $P$ need not be a pants decomposition of $S$.  Observe, however, that any curve $\alpha \subset S \setminus P$ is not $H$-symmetric, otherwise it would descend to a curve on $\OO$ disjoint from $\widebar{P}$.\\

By the above observation, the components of $S \setminus P$ are pairs of pants and subsurfaces, $Y \subset S\setminus P$, which are stabilized by $H$.  For any such $Y$, the action of $H$ restricts to an action on $Y$.  Since $Y$ supports no symmetric curves, we must have that the quotient of $Y$ by $H|_Y$, $Y/(H|_Y)$, is a pair of orbipants, which we note has a unique hyperbolic structure once the lengths of any pants curves are chosen.  In particular, this means that the fixed point set in each such $\TT(Y)$ is a single point.\\

Let $\widebar{X} \in V_{\delta}(\widebar{P}) \subset \TT(\OO)$ and consider its unique lift $X \in \Fix(H)$.  Consider the stratum $\TT(S, P) \subset \widebar{\TT(S)}$, where all curves in $P$ have been pinched to nodes.  Since $\TT(S,P)$ is convex and $(\widebar{\TT(S)}, d_{\widebar{WP}})$ is a complete CAT(0) space, it follows from \cite[Proposition II.2.4]{BH99} that there is a unique closest point $X_P \in \TT(S, P)$ to $X$ in $\TT(S,P)$.\\

Recall that the action of $\MCG(S)$ extends to $(\widebar{\TT(S)}, d_{\widebar{WP}})$ and observe that $H$ stabilizes $\TT(S,P)$ because its defining curves are $H$-symmetric.  Since $X \in \Fix(H)$ and $X_P$ is the closest point to $X$ in $\TT(S,P)$, it follows that $X_P$ must also be fixed by $H$.\\

We claim that $X_P$ is the only point in $\TT(S,P)$ fixed by $H$.  To see this, recall that $X_P$ is defined by a point  in $\TT(Y)$ for each nonpants component $Y \subset S \setminus P$.  Since $X_P$ is fixed by the action of $H$, it follows that the points in the $\TT(Y)$ which define $X_P$ must also be fixed by $H$.  As observed above, each such $\TT(Y)$ has a unique point fixed by $H$.  As such, $X_P$ is the unique point in $\TT(S,P)$ fixed by $H$.\\

Wolpert's Theorem \ref{r:Wol est} implies that 

\[d_{\widebar{WP}}(X, X_P) =d_{\widebar{WP}}(X, \TT(S,P)) \asymp_{\delta} 1\]
as $X_P$ was the closest point in $\TT(S,P)$ to $X$.\\

Let $X' \in \Fix(H) \cap V_{\delta}(P)$ be different from $X$.  Since our choice of $X$ was arbitrary, it follows that $X_P$ is also the closest point to $X'$ in $\TT(S,P)$ and so
\[d_{\widebar{WP}}(X', X_P) \asymp_{\delta} 1\]

Thus the triangle inequality implies that 
\[d_{WP}(X,X') = d_{\widebar{WP}}(X, X') \asymp_{\delta} 1\]

\end{proof}

\section{Almost-fixed points are close to fixed points} \label{r:afp section}

This section is devoted to proving the Main Theorem \ref{r:main}.\\

The outline of the proof of Theorem \ref{r:main} is as follows: Beginning with any almost-fixed point $\sigma \in \Fix_R(H) \subset \TT(S)$, we first use the nonpositive curvature of $\TT(S)$ with the Weil-Petersson metric and work of Wolpert to find a fixed point, $X\in \TT(S)$.  Applying results of Brock, Masur-Minsky, Rafi, and the author, we deduce that the Teichm\"uller distance of $X$ to $\sigma$ is coarsely determined by large projections to horoballs.  Using a characterization of the short curves for the barycenter developed in Lemma \ref{r:afp sym short}, we apply Proposition \ref{r:VP orb} and results of Minsky, Rafi, Wolpert, and the author to show in Proposition \ref{r:reduce short curves} that the large projections to horoballs can be reduced to large projections to annuli.  It follows from Tao's Lemma \ref{r:sym links} that these annular large links can be grouped into symmetric families which come with an ordering from the hierarchy machinery.  The proof of Theorem \ref{r:main} describes how to leap across the symmetric families one at a time by applying $H$-symmetric multitwists, while staying in $\Fix(H)$ at each step.  This process ends with new fixed point whose distance to $\sigma$ is bounded as a function of $R$ and the topology of $S$, thus completing the proof.

\subsection{The Teichm\"uller geometry of Weil-Petersson barycenters} \label{r:short bary}

In this subsection, we analyze the short curves of the Weil-Petersson barycenter of an $H$-orbit of an almost-fixed point.  First, we recall a basic result of CAT(0) geometry, as recorded in \cite[Proposition II.2.7]{BH99}:

\begin{lemma}\label{r:cat bary}
Let $X$ be a complete $\mathrm{CAT(0)}$ space.  If $Y \subset X$ is a bounded set of radius $R$, then there exists a unique point $C \in X$, the \emph{barycenter} of $Y$, such that $Y \subset \widebar{B}(C,R)$.
\end{lemma}

Fix $R_0>0$ and let $\tau \in \Fix^T_{R_0}(H)$.  It follows from Theorem \ref{r:aug qi} that there is an $\widetilde{R}>0$ depending only on $R_0$ and $S$ such that $\tilde{\mu}_{\tau} \in \widetilde{\Fix}_{\widetilde{R}}(H)$.  Since the Weil-Peterrson metric is coarsely dominated by the Teichm\"uller metric (\cite{Lin74}; see Remark \ref{r:wp<t}), it follows that there is an $R=R(R_0)>0$ for which $\tau \in \Fix^{WP}_R(H)$, where $R$ only differs from $R_0$ by a multiplicative constant.  Since the augmented Teichm\"uller space, $\widebar{\TT(S)}$, is a complete CAT(0) space, it follows from Lemma \ref{r:cat bary} that the $H$-orbit of $\sigma$ has a barycenter $X'_0 \in \widebar{\Fix(H)} \subset \widebar{\TT(S)}$ in the Weil-Petersson metric, where $\widebar{\Fix(H)}$ is the completion of $\Fix(H)$ to $\widebar{\TT(S)}$, namely marked noded surfaces which are preserved by the action of $H$.\\
In the case that $X'_0 \in \widebar{\TT(S)} \setminus \TT(S)$, the next lemma produces a new fixed point $X_0 \in \Fix(H)$ arbitrarily close to $X'_0 \in \Fix(H)$ in $d_{\widebar{WP}}$, the extension of the Weil-Petersson metric to $\widebar{\TT(S)}$:

\begin{lemma}
For any $\delta>0$, there is a point $X_0 \in \Fix(H)\subset \TT(S)$ with $d_{WP}(X_0, X'_0) \leq \delta$.
\end{lemma}

\begin{proof}
If $X'_0 \in \Fix(H)\subset \TT(S)$, then we may choose $X_0 = X'_0$.\\

If not, then $X'_0$ has some simplex of curves $\alpha \subset \CC(S)$, each of whose constituent curves has been pinched down to a node.  Since $X'_0 \in \widebar{\Fix(H)}$, it follows that $H$ preserves $\alpha$.  That is, $\alpha$ is $H$-symmetric.  Let $\widebar{\alpha} \subset \CC(\OO)$ be the simplex which lifts to $\alpha$.\\

Let $Y \in \Fix(H)$ be any other fixed point and consider the unique, finite Weil-Petersson geodesic ray emanating from $Y$ and terminating at $X'_0$, which we denote by $\mathcal{G}$.  Since the action of $\MCG(S)$ extends to the completion $\widebar{\TT(S)}$, it follows that $\mathcal{G}$ is fixed by $H$.   Since $\mathcal{G}$ has finite length, we can let $X_0 \in \mathcal{G}$ be any point satisfying $d_{\widebar{WP}}(X_0, X_0')<\delta$, completing the proof.

\end{proof}

For any $\epsilon'>0$, denote by $\Lambda_{\epsilon',\tau}$ the set of curves for which $l_{\tau}(\lambda)<\epsilon'$.  Recall that in Definition \ref{r:epsilon} we fixed $\epsilon_0>0$ so that Minsky's Product Regions Theorems \ref{r:product} holds.  The following lemma says that if $\tau$ has a really short curve, then each curve in the $H$-orbit of given curve must have $\tau$-length less than $\epsilon_0$.  In particular, the whole orbit must be in the base of $\tilde{\mu}_{\tau}$, a shortest augmented marking for $\tau$.

\begin{lemma}[Almost-fixed points have symmetric short curves]\label{r:afp sym short}
There exists  $\epsilon''>0$ sufficiently small, so that if $\lambda \in \Lambda_{\epsilon'',\tau}$, then $\lambda$ is $H$-symmetric and $H \cdot \lambda \subset \Lambda_{\epsilon_0,\tau}$.
\end{lemma}

\begin{proof}
Consider a shortest augmented marking $\tilde{\mu}_{\tau} \in \AM(S)$ (see Subsection \ref{r:interplay} for the definition of $\tilde{\mu}_{\tau}$).  Since $\tau \in \Fix^T_{R_0}(H)$, recall that Theorem \ref{r:aug qi} implies that there is an $\widetilde{R}>0$ depending only on $R_0$ and $S$ such that $\tilde{\mu}_{\tau} \in \widetilde{\Fix}_{\widetilde{R}}(H)$.\\

Recall from Subsection \ref{r:AMS section} that to each curve $\alpha \in \mathrm{base}(\tilde{\mu}_{\tau})$, we assign a length $D_{\alpha}(\tilde{\mu}_{\tau})$, the coordinate which coarsely represents how short $\alpha$ is in $\tilde{\mu}_{\tau}$.  See Subsection \ref{r:AMS section} for the definition of $D_{\alpha}$ and Subsection \ref{r:interplay} for how it is defined from a point in $\TT(S)$.\\

Let $\epsilon_1''>0$ be small enough so that if $\lambda \in \Lambda_{\epsilon_1'', \tau}$, then $D_{\lambda}(\tilde{\mu}_{\tau})>\widetilde{R} + M_1$, where $M_1$ is the constant from Remark \ref{r:coarsely natural} (see Subsection \ref{r:interplay} for why short curves have large length coordinates).  If $\lambda$ is not $H$-symmetric, then there is some $h \in H$ such that $i(\lambda, h\cdot \lambda) \geq 1$.  Since $D_{h\cdot\lambda} (\tilde{\mu}_{h\cdot \tau}) > \widetilde{R}$, it follows that
\[d_{\AM(S)}(\tilde{\mu}_{\tau},\tilde{\mu}_{h\cdot \tau}) \geq  d_{\AM(S)}(\tilde{\mu}_{\tau},h\cdot \tilde{\mu}_{\tau}) +M_1 \geq d_{\HHH_{\lambda}}(\tilde{\mu}_{\tau},h\cdot \tilde{\mu}_{\tau}) + d_{\HHH_{h\cdot \lambda}}(\tilde{\mu}_{\tau},h\cdot \tilde{\mu}_{\tau})+M_1 > 2\widetilde{R}+M_1\]
a contradiction of $\tilde{\mu}_{\tau} \in \widetilde{\Fix}_{\widetilde{R}}(H)$.  The first inequality follows from Remark \ref{r:coarsely natural}.  The second inequality follows from the fact that any path from any augmented marking with $D_{\gamma_1}>0$ to one with $D_{\gamma_2}>0$ for $i(\gamma_1, \gamma_2) >0$ must completely exit $\HHH_{\gamma_1}$ before entering $\HHH_{\gamma_2}$, at a cost of at least $D_{\gamma_1} + D_{\gamma_2}$.\\

Now suppose there is an $h \in H$ such that $h\cdot \lambda \notin \Lambda_{\epsilon_0,\tau}$.  It follows that $\lambda \notin \Lambda_{\epsilon_0,h^{-1}\cdot \tau}$ and $D_{\lambda}(\tilde{\mu}_{h^{-1}\tau})=0$.  For sufficiently small $\epsilon_2''>0$, we have $d_{\HHH_{\lambda}}(\tilde{\mu}_{\tau},\tilde{\mu}_{h^{-1}\tau})> A \cdot \widetilde{R} + B+M_1$, where $A, B$ are the constants depending only on $S$ from Theorem \ref{r:distances} and $M>0$ is again the constant from Remark \ref{r:coarsely natural}.  Theorem \ref{r:distances} implies that $d_{\AM(S)}(\tilde{\mu}_{\tau}, \tilde{\mu}_{h^{-1}\tau}) > \widetilde{R}$, a contradiction of the fact that $\tilde{\mu}_{\tau} \in \widetilde{\Fix}_{\widetilde{R}}(H)$.\\

Choosing $\epsilon''< \min\{\epsilon''_1, \epsilon''_2\}$ satisfies both of the above arguments, completing the proof.
\end{proof}

Consider the subset of $\TT(S)$ of metrics in which all curves in $\Lambda_{\epsilon'', \tau}$ are shorter than $\epsilon_0$:

\[V_{\epsilon_0}(\Lambda_{\epsilon'', \tau}) = \{Y \in \TT(S) | l_Y(\lambda) < \epsilon_0, \forall \lambda \in \Lambda_{\epsilon'', \tau}\}\]

Equivalently, $V_{\epsilon_0}(\Lambda_{\epsilon'', \tau})$ contains all points in $\TT(S)$ whose shortest augmented markings contain $\Lambda_{\epsilon'',\tau}$ in their bases.  By WP-convexity of length functions, $V_{\epsilon_0}(\Lambda_{\epsilon'', \tau})$ is WP-convex.  Lemma \ref{r:afp sym short} implies that $H \cdot \tau \in V_{\epsilon_0}(\Lambda_{\epsilon'', \tau})$.  Since $H \cdot \tau \subset \widebar{B}(X_0, R_0)$ (see Lemma \ref{r:cat bary}), it follows from the convexity of length functions that $X_0 \in V_{\epsilon_0}(\Lambda_{\epsilon'', \tau})$.  This implies that $\Lambda_{\epsilon'', \tau} \subset \Lambda_{\epsilon_0,X_0}$ and, in particular, that $\Lambda_{\epsilon'', \tau} \subset \mathrm{base}(\tilde{\mu}_{X_0})$.  As $X_0 \in \Fix(H)$, it follows that $H \cdot \Lambda_{\epsilon'', \tau} \subset \mathrm{base}(\tilde{\mu}_{X_0})$.  That is, the full $H$-orbits of all of $\tau$'s really short curves are also short in $X_0$.\\

Our goal in Proposition \ref{r:reduce short curves} below is to remove the combinatorial complexity between $X_0$ and $\tau$ coming from the short curves of $X_0$, which can come in the form of both the length of and twisting about these curves.\\

Let $\tilde{\mu}_1, \tilde{\mu}_2 \in \AM(S)$ be any two augmented markings.  For each $\alpha$, let $n_{\alpha} = d_{\alpha}(\tilde{\mu}_1, \tilde{\mu}_2)$, so that $d_{\alpha}(\tilde{\mu}_1, T_{\alpha}^{\pm n_{\alpha}}\cdot \tilde{\mu}_2) <C$, where $T_{\alpha}$ denote the right Dehn (half)twist about $\alpha$ and $C$ depends only on $S$.  Then 
\[d_{\widehat{\HHH}_{\alpha}}(\tilde{\mu}_1, T_{\alpha}^{\pm n_{\alpha}} \tilde{\mu}_2) <|D_{\alpha}(\tilde{\mu}_1) - D_{\alpha}(\tilde{\mu}_2)| + 2C\]

Now suppose there is a constant $D>0$ such that $|D_{\alpha}(\tilde{\mu}_1) - D_{\alpha}(\tilde{\mu}_2)| < D$, for $\alpha \in \CC(S)$.    Then there is a $D'$ which depends only on $D$ and $S$ such that
\[d_{\widehat{\HHH}_{\alpha}}(\tilde{\mu}_1, T_{\alpha}^{\pm n_{\alpha}} \tilde{\mu}_2) < D'\]

We are now ready to state and prove Proposition \ref{r:reduce short curves}, a key technical step on the way to the proof of the Main Theorem \ref{r:main}.  In it, we produce a new fixed point, $X\in \Fix(H)$, whose Weil-Petersson distance to $\tau$ is still uniformly bounded, but whose Teichm\"uller distance has decreased in two significant ways: $X$ and $\tau$ have uniformly bounded projections to horoballs coming from the short curves $X$ inherits from $\tau$, $H \cdot \Lambda_{\epsilon'', \tau}$, and $X$ and $\tau$ have uniformly bounded projections to horoballs coming from the short curves of $X$ which it does not inherit from $\tau$, $\Lambda_{\epsilon, X_0} \setminus \left(H \cdot \Lambda_{\epsilon'', \tau}\right)$.  In the proof, we create a new, preliminary fixed point $X' \in \Fix(H)$, whose coarse lengths for curves short in $X_0$ are coarsely equal.  Then we apply a carefully chosen combination of multitwists to $X'$ to obtain a new fixed point $X \in \Fix(H)$, whose twisting coordinates about the short curves of $X_0$ are coarsely equal to those of $\tau$.  As we show in Lemma \ref{r:unif bound} below, the end result is that the Teichm\"uller distance between $X$ and $\tau$ is coarsely determined by projections to a uniformly bounded number of annuli, which is a significant reduction of the combinatorial complexity between $\tau$ and $\Fix(H)$.

\begin{proposition}[Reducing short curves] \label{r:reduce short curves}
There is a fixed point $X \in \Fix(H)$ with shortest augmented marking $\tilde{\mu}_X \in \AM(S)$ which has the following properties:
\begin{enumerate}
\item For every $\alpha \in \CC(S)$, we have $D_{\alpha}(\tilde{\mu}_X) \overset{+}{\asymp}_R D_{\alpha}(\tilde{\mu}_{\tau})$ \label{r:rsc 1}
\item For any $\alpha \in \Lambda_{\epsilon_0, X_0}$, we have $d_{\widehat{\HHH}_{\alpha}}(\tilde{\mu}_X, \tilde{\mu}_{\tau}) \asymp_R 1$ \label{r:rsc 2}
\item For any nonannular $Y \subset S$, we have $d_Y(\tilde{\mu}_X,\tilde{\mu}_{\tau}) \asymp_R 1$, and so $d_{WP}(X, \tau) \prec \widetilde{R}$\label{r:rsc 3}
\end{enumerate}
\end{proposition}

\begin{proof}

Let $\widebar{\Lambda}_{\epsilon_0,X_0} \subset \CC(\OO)$ be the curves which lift to $\Lambda_{\epsilon_0, X_0} \subset \CC(S)$.  The comments following Lemma \ref{r:afp sym short} imply that $\widebar{\Lambda}_{\epsilon'',\tau} \subset \widebar{\Lambda}_{\epsilon_0,X_0}$.  The key initial observation, which follows from Lemmas \ref{r:lifted coordinates} and \ref{r:afp sym short} and the remarks which follow the latter, is that $\tau, X_0 \in V_{\epsilon_0}(H \cdot \Lambda_{\epsilon'',\tau}) \subset V_L(H \cdot \Lambda_{\epsilon'',\tau})$, with the latter inclusion following from our choice of $\epsilon_0$ in Definition \ref{r:epsilon}.\\

It follows from the proof of Lemma \ref{r:afp sym short} that $D_{\alpha}(\tilde{\mu}_{\tau}) \asymp_R 0$ for all curves $\alpha \subset S\setminus \left(H \cdot \Lambda_{\epsilon'',\tau}\right)$ (see Subsection \ref{r:AMS section} for the definition of $D_{\alpha}$).  Since $H \cdot \Lambda_{\epsilon'',\tau} \subset \Lambda_{\epsilon_0,X_0}$, in order to build a fixed point which satisfies conclusion (\ref{r:rsc 1}), it suffices to adjust the $D_{\lambda}(\tilde{\mu}_{X_0})$ to within bounded distance from $D_{\lambda}(\tilde{\mu}_{\tau})$ for $\lambda \in H \cdot \Lambda_{\epsilon'',\tau}$, and to adjust $D_{\lambda}(\tilde{\mu}_{X_0})$ to $0$ for $\lambda \in \Lambda_{\epsilon_0}(\epsilon_0, X_0)$.  This is done directly in Fenchel-Nielsen coordinates on $\TT(\OO)$.  Proposition \ref{r:VP orb} then will imply conclusion \ref{r:rsc 3}.  Finally, we arrive at conclusion \ref{r:rsc 2} by applying appropriate multitwists to the new point we build.\\

Complete $\widebar{\Lambda}_{\epsilon_0, X_0}$ to a Bers orbipants decomposition for $X_0$, $\widebar{P}_{\widebar{\Lambda}_{\epsilon_0, X_0}} \in \PP(\OO)$; that is, $l_{\widebar{X}_0}(\alpha) < L'$ for all $\widebar{\alpha} \in \widebar{P}_{\widebar{\Lambda}_{\epsilon_0, X_0}}$, where $L'>0$ is the constant from Corollary \ref{r:bers orb}.  Recall from Lemma \ref{r:lifted coordinates} that $\widebar{P}_{\widebar{\Lambda}_{\epsilon_0, X_0}}$ lifts to an $H$-symmetric partial pants decomposition on $S$, $P_{{\Lambda}_{\epsilon_0, X_0}}$, which we can extend to a full pants decomposition $P_0 \in \PP(S)$.  Fix Fenchel-Nielsen coordinates for $\TT(S)$ based on $P_0$.\\

For each orbit of curve in $\Lambda_{\epsilon_0,X_0}$, fix a representative $\lambda$ which lifts from $\widebar{\lambda} \in \CC(\OO)$.  Let $\widebar{X'} \in \TT(\OO)$ be any point whose length coordinates with respect to $\widebar{P}_{\widebar{\Lambda}_{\epsilon_0, X_0}}$ satisfy the following conditions:

\begin{enumerate}
\item $l_{\widebar{X}'}(\widebar{\lambda}) = l_{\tau}(\lambda) \cdot \frac{1}{N_{\widebar{\lambda}}} < \epsilon_0$ for each orbit representative $\lambda \in H\cdot  \Lambda_{\epsilon'', \tau}$, where $N_{\widebar{\lambda}}$ is the constant from Lemma \ref{r:lifted coordinates}\label{r:reduce 1}
\item $l_{\widebar{X}'}(\widebar{\gamma}) = \epsilon_0$ for each orbit representative $\gamma \in \Lambda_{\epsilon_0,X_0} \setminus (H \cdot \Lambda_{\epsilon'', \tau})$ \label{r:reduce 2}
\item $l_{\widebar{X}'}(\widebar{\alpha}) = l_{\widebar{X}_0}(\widebar{\alpha})$ for every other $\widebar{\alpha} \in \widebar{P}_{\widebar{\Lambda}_{\epsilon_0, X_0}} \setminus \widebar{\Lambda}_{\epsilon_0, X_0}$. \label{r:reduce 3}
\end{enumerate}

We claim the lift $X' \in \Fix(H)$ of any such $\widebar{X'} \in \TT(\OO)$ satisfies conclusion (\ref{r:rsc 1}).\\

To see this, first observe that condition (\ref{r:reduce 1}) implies that $D_{\alpha}(\tilde{\mu}_X') \asymp_R D_{\alpha}(\tilde{\mu}_{\tau})$ for any $\alpha \in H \cdot \Lambda_{\epsilon'', \tau}$, as the $N_{\widebar{\lambda}}$ are uniformly bounded by Lemma \ref{r:lifted coordinates}.  Next, since Lemma \ref{r:afp sym short} implies that $D_{\alpha}(\tilde{\mu}_{\tau}) \asymp 0$ for all curves $\alpha \subset S\setminus \left(H \cdot \Lambda_{\epsilon'',\tau}\right)$, conditions (\ref{r:reduce 2}) and (\ref{r:reduce 3}), and the fact that $X' \in \Fix(H)$ so that any $\alpha \in S \setminus P_{\widebar{\Lambda}_{\epsilon_0,X_0}}$ are necessarily not $H$-symmetric and thus cannot be short in $X'$, imply that $D_{\alpha}(\tilde{\mu}_{X'}) \asymp_R D_{\alpha}(\tilde{\mu}_{\tau})$ for all such $\alpha \subset S\setminus \left(H \cdot \Lambda_{\epsilon'',\tau}\right)$.  Finally, since $X', \tau \in V_{\epsilon''}(H \cdot \Lambda_{\epsilon'', \tau})$, we have $D_{\alpha}(\tilde{\mu}_{X'}) = D_{\alpha}(\tilde{\mu}_{\tau}) = 0$, for all $\alpha \pitchfork H \cdot \Lambda_{\epsilon'', \tau}$, by the Collar Lemma.  Thus conclusion (\ref{r:rsc 1}) holds for $X'$.\\

It follows from its definition that $\widebar{X}' \in V_{L'}(\widebar{P}_{\widebar{\Lambda}_{\epsilon_0, X_0}} )$.  As $\widebar{X}_0 \in V_{L'}(\widebar{P}_{\widebar{\Lambda}_{\epsilon_0, X_0}} )$, conclusion (\ref{r:rsc 3}) for $X'$ follows from Proposition \ref{r:VP orb} and the triangle inequality.\\

Generically, $X'$ does not satisfy conclusion (\ref{r:rsc 2}).  To build a point which does, we apply some carefully chosen $H$-symmetric multitwists to reduce the annular projections between $X'$ and $\tau$.  We then prove that the resulting point still satisfies conclusions (\ref{r:rsc 1}) and (\ref{r:rsc 3}).\\

Let $\widebar{\Lambda}_{\epsilon_0, X_0} \subset \CC(\OO)$ be the set of curves which lift to $H \cdot \Lambda_{\epsilon_0, X_0} \subset \CC(S)$.  Suppose that $\widebar{\Lambda}_{\epsilon_0, X_0} $ consists of $N_{\tau}$ different $H$-orbits of curves and decompose it into these orbits, 

\[\widebar{\Lambda}_{\epsilon_0, X_0} = \{\lambda_{1,1}, \dots, \lambda_{1, m_1}, \dots, \lambda_{N_{\tau},1}, \dots,\lambda_{N_{\tau}, m_{N_{\tau}}}\}\]

Note that both the $m_i$ and $N_{\tau}$ are uniformly bounded.\\

For each $i$, let $T_{\lambda_i} = \prod_{j=1}^{m_i} T_{\lambda_{i,j}}^{(-1)^{s_i} \cdot d_i}$, where $T_{\lambda_{i,j}}$ is the Dehn (half)twist around $\lambda_{i,j}$, $d_i = d_{\lambda_i,1}(\tilde{\mu}_{X'},\tilde{\mu}_{\tau})$, and the sign $s_i$ depends on whether $\pi_{\lambda_{i,1}}(\tilde{\mu}_{X'})$ differs from $\pi_{\lambda_{i,1}}(\tilde{\mu}_{\tau})$ by right or left Dehn (half)twists around $\lambda_{i,1}$.\\

Set $T_{\Lambda_{\epsilon_0,X_0}} = \prod_{i=1}^{N_{\tau}} T_{\lambda_i}$ and $X = T_{\Lambda_{\epsilon_0, X_0}} \cdot X'$.  We claim that $X'$ satisfies the conclusions of the proposition.\\

First, observe that since $\Lambda_{\epsilon_0, X_0}$ is an $H$-symmetric multicurve, $T_{\Lambda_{\epsilon_0,X_0}} \in C_{\MCG(S)}(H)$, the centralizer of $H$ in $\MCG(S)$, which is contained in the normalizer of $H$, which stabilizes $\Fix(H)$.  Thus $X \in \Fix(H)$.\\

Second, since $\Lambda_{\epsilon_0, X_0} \subset \mathrm{base}(\tilde{\mu}_{X'}) \cap \mathrm{base}(\tilde{\mu}_{X})$, it follows that $d_Y(\tilde{\mu}_{X}, \tilde{\mu}_{X'}) \asymp 1$ uniformly for any $Y \subset S$ not an annulus over a curve in $\Lambda_{\epsilon_0,X_0}$.   Because $T_{\Lambda_{\epsilon_0, X_0}}$ preserves the curves in $\Lambda_{\epsilon_0, X_0}$ and any curves disjoint from them, namely $P_0$, conclusions (\ref{r:rsc 1}) and (\ref{r:rsc 3}) hold for $X$.\\

Finally, observe that Lemma \ref{r:sym sub proj} implies that $d_{\lambda_{i,j}}(\tilde{\mu}_{X_0}, \tilde{\mu}_{\tau}) \asymp_R d_{\lambda_{i, k}}(\tilde{\mu}_{X_0}, \tilde{\mu}_{\tau})$ for any $j,k$.  Thus the choice of $T_{\Lambda_{\epsilon_0, X_0}}$ and the triangle inequality imply that $d_{\alpha}(\tilde{\mu}_X, \tilde{\mu}_{\tau}) \asymp_R 1$ for each $\alpha \in \Lambda_{\epsilon_0, X_0}$.  Since conclusion (\ref{r:rsc 1}) also holds for $X$ for each $\alpha \in \Lambda_{\epsilon_0,X_0}$, it follows that conclusion (\ref{r:rsc 2}) holds for $X$.  This completes the proof.

\end{proof}
\subsection{Proof of the main theorem}

Recall our main goal of this section, achieved in Theorem \ref{r:main} below, is to find a fixed point whose distance to $\tau \in \Fix_R^T(H)$ is bounded in terms of $R$ and $S$.  Proposition \ref{r:reduce short curves} produces a fixed point $X \in \Fix(H)$ which has the same very short curves as $\tau$, whose distance to $\tau$ in any horoball over any of these short curves is uniformly bounded, and whose distance in any other nonhoroball subsurface is uniformly bounded.  Before proceeding with the proof of Theorem \ref{r:main}, we analyze and organize the remaining large horoball projections.\\

Observe that $X$ and $\tau$ have  $H\cdot \Lambda_{\epsilon'', \tau}$ as short curves, so $X, \tau \in Thin_{\epsilon, S}(\Lambda_{\epsilon'', \tau}) \asymp Q(\Lambda_{\epsilon'', \tau})$.  By Corollary \ref{r:distance prod},
\[d_{\AM(S)}(\tilde{\mu}_X, \tilde{\mu}_{\tau}) \asymp \sum_{\alpha \in \CC(S\setminus H \cdot \Lambda_{\epsilon'', \tau})} \left[d_{\HHH_{\alpha}}(\tilde{\mu}_X, \tilde{\mu}_{\tau})\right]_K\]
Since $d_{\HHH_{\lambda}}(\tilde{\mu}_X, \tilde{\mu}_{\tau}) \prec \widetilde{R}$ for all $\lambda \in \Lambda_{\epsilon, X_0}$ by Proposition \ref{r:reduce short curves}, we have
\[
d_{\AM(S)}(\tilde{\mu}_X, \tilde{\mu}_{\tau}) \asymp \sum_{\alpha \in \CC(S\setminus \Lambda_{\epsilon, X_0})} \left[d_{\HHH_{\alpha}}(\tilde{\mu}_X, \tilde{\mu}_{\tau})\right]_K\]

Recall that the very short curves of $\tau$, $\Lambda_{\epsilon'', \tau}$, are a subset of the short curves of $X$, $\Lambda_{\epsilon, X} = \Lambda_{\epsilon, X_0}$.  Because there is a uniform bound on the distance between the projections of $\tau$ and $X$ to any horoball over a curve in $\Lambda_{\epsilon, X}$, it follows that there is a lower bound on the $\tau$- and $X$-lengths of any curve not in $\Lambda_{\epsilon, X}$.  Thus the projections of $\tilde{\mu}_{\tau}$ and $\tilde{\mu}_X$ to any other combinatorial horoball have uniformly bounded length coordinates and the sum becomes

\begin{equation}
d_{\AM(S)}(\tilde{\mu}_X, \tilde{\mu}_{\tau}) \asymp \sum_{\alpha \in \CC(S\setminus \Lambda_{\epsilon, X_0})} \left[\log d_{\alpha}(\tilde{\mu}_X, \tilde{\mu}_{\tau})\right]_K \label{r:simple eq}
\end{equation}

\begin{lemma} \label{r:unif bound}
The number of terms which can appear in the sum of (\ref{r:simple eq}) is uniformly bounded.
\end{lemma}

\begin{proof}

Let $\widetilde{\Gamma}$ be an augmented hierarchy path between $\tilde{\mu}_X$ and $\tilde{\mu}_{\tau}$ based on a hierarchy $J$ (see Subsection \ref{r:hier sub}).  Observe that the number of curves appearing as base curves of augmented markings in $\widetilde{\Gamma}$ is determined by the number of flip moves in $\widetilde{\Gamma}$.  Since each such flip move makes progress along some $g_Y \in J$, for some nonannular $Y \subset S$, it follows that if there is not a bound on the number of base curves appearing in $\widetilde{\Gamma}$, then there is not a bound on either the length of geodesics in $J$ or the number of nonannular subsurfaces supporting geodesics in $J$.  Both imply that $d_Y(\tilde{\mu}_X, \tilde{\mu}_{\tau})$ is unbounded for some nonannular $Y \subset S$ (possibly $S$ itself), which contradicts the fact that $\tilde{\mu}_X$ and $\tilde{\mu}_{\tau}$ have bounded nonannular subsurface projections.  The bound on the number of curves appearing in the sum of (\ref{r:simple eq}) is uniform because the bound on the subsurface projections is uniform, depending only on $S$ and the almost-fixed constant $R$.\\
\end{proof}

We are now ready to prove the main theorem.

\begin{theorem} [Almost-fixed points are close to fixed points] \label{r:main}
For any $R>0$, there is an $R'=R'(R,S)>0$ such that the following holds.  Let $H \leq \MCG(S)$ be a finite subgroup and $\Fix(H)\subset \TT(S)$ its fixed point set.  For any $\tau \in \Fix^T_R(H)$, there is fixed point $ \sigma \in \Fix(H)$ such that $d_T(\tau,\sigma) <R'$.
\end{theorem}

\begin{proof}
Let $X\in \TT(S)$ be as in Proposition \ref{r:reduce short curves}.  As the constant $\widetilde{R}$ in Proposition \ref{r:reduce short curves} was a constant depending on $R$, we have shown that 
\begin{equation}
d_{T}(X, \tau) \asymp_{R} \sum_{\alpha \in \CC(S\setminus \Lambda_{\epsilon, X_0})} \left[ \log d_{\alpha}(\tilde{\mu}_X, \tilde{\mu}_{\tau})\right]_K \label{r:ann sum 1}
\end{equation}

More precisely, Proposition \ref{r:reduce short curves} states that $d_Y(\tilde{\mu}_X, \tilde{\mu}_{\tau})<K$ for any nonannular subsurface $Y \subset S \setminus \Lambda_{\epsilon, X_0}$, where $K$ is a constant depending only on $R$ and $S$.\\

We now organize the $\alpha$ that have nonzero terms in equation (\ref{r:ann sum 1}).  By Tao's Lemma \ref{r:sym links}, if we increase the large link threshold to $\widehat{K} = \widehat{K}(R,S)>0$, then these annuli are $H$-symmetric and we can group them into their $H$-orbits, $\mathcal{A}= \{\mathcal{A}_1, \dots,\mathcal{A}_N\}$, where $\mathcal{A}_i$ is the $H$-orbit of $\alpha_i$.\\

We note that $N$ is uniformly bounded because the number of annuli appearing in the sum is uniformly bounded, by Lemma \ref{r:unif bound}.\\

Let $\Gamma_{X,\tau}$ be any augmented hierarchy path from $\tilde{\mu}_X$ to $\tilde{\mu}_{\tau}$.  By rearranging, we may assume that the order of the indices of the $\alpha_i$ coincides with the order of appearance of the $\alpha_i$ along $\Gamma_{X,\tau}$.  Note that Lemma \ref{r:sym links} implies that the curves within each symmetric family, $\mathcal{A}_i$, are not time-ordered.\\

We now apply the tools developed in Subsection \ref{r:coarse proj section}.  Recall that for a simplex $\Delta \subset \CC(S)$, $Q(\Delta) = \{\tilde{\mu} \in \AM(S)| \Delta \subset \mathrm{base}(\tilde{\mu})\}$ and $\phi_{\Delta}: \AM(S) \rightarrow Q(\Delta)$ was the closest point projection.\\

In what follows, we explain how to create a sequence of fixed points $X_1, \dots, X_N \in \Fix(H)$, with $d_{\TT(S)}(X_N, \tau) \asymp_R 1$, where $N$ is again the number of symmetric families of annuli in $\mathcal{A}$.  The $(i+1)$-step begins with projecting $\tilde{\mu}_{X_i}$, a shortest augmented marking for $X_i$, to $Q(\mathcal{A}_{i+1})$  and showing that this projection is uniformly close to $\tilde{\mu}_{X_i}$.  We then apply a large $H$-symmetric multitwist around the curves in $\mathcal{A}_{i+1}$ to both $X_i$ and its projection to $Q(\mathcal{A}_i)$, the latter of which we show has made the progress toward $\tau$ that we want, with the former coming along for the ride and whose image we call $X_{i+1}$.  This multitwisting process is identical to the process at the end of the proof of Proposition \ref{r:reduce short curves}, but now the $X_i$ need not be in a obviously good place to apply the $(i+1)^{\mathrm{th}}$-group of multitwists.  The key observation is that $d_{\AM(S)}(\tilde{\mu}_{X_i}, \phi_{\mathcal{A}_{i+1}}(\tilde{\mu}_{X_i})) \asymp_R 1$ for all $i$, a fact which requires understanding the subsurfaces through which $\Gamma_{X,\tau}$ passes.  Showing that $d_{\AM(S)}(X_N, \tilde{\mu}_{\tau}) \asymp_R 1$ then involves comparing subsurface projections and showing that projections to horoballs over curves in $\mathcal{A}$ have changed a significant amount, in particular moving them close to those for $\tilde{\mu}_{\tau}$.\\

Let $\tilde{\mu}_X \in \AM(S)$ be a shortest augmented marking for $X$.  We begin by projecting $\tilde{\mu}_X$ to $Q(\mathcal{A}_1)$.  Set $\tilde{\mu}_{\alpha_1} = \phi_{\mathcal{A}_1}(\tilde{\mu}_X)$.\\

\underline{Claim 1}: $d_{\AM(S)}(\tilde{\mu}_X, \tilde{\mu}_{\alpha_1}) \asymp 1$.\\

Before we prove the claim, we introduce some notation to simplify our calculations.  For each $i$, label the curves in $\mathcal{A}_1= \{\alpha_{i,1}, \dots, \alpha_{i,n_i}\}$.  We note that each $n_i$ satisfies $n_i \leq |H|$.\\

First, we prove that for all $j$, $d_{\AM(S)}\left(\tilde{\mu}_X,\phi_{\alpha_{1,j}}(\tilde{\mu}_{X})\right) \asymp 1$.  To see this, note that Lemma \ref{r:cpj} implies that $\phi_{\alpha_{1,j}}$ is coarsely a closest point projection to $Q(\alpha_{1,j})$, so that 
\[d_{\AM(S)}(\tilde{\mu}_X, \phi_{\alpha_{1,j}}(\tilde{\mu}_X)) \asymp \sum_{Y \pitchfork \alpha_{1,j}} \left[d_Y(\tilde{\mu}_X, \phi_{\alpha_{1,j}}(\tilde{\mu}_X))\right]_{L_1}\]

and 

\[\sum_{Y \subset S \setminus \alpha_{1,j}} \left[d_Y(\tilde{\mu}_X, \phi_{\alpha_{1,j}}(\tilde{\mu}_X))\right]_{L_1} =0\]
where $L_1$ is the uniform constant from Lemma \ref{r:cpj}.\\

In order to show that $d_{\AM(S)}(\tilde{\mu}_X, \phi_{\alpha_{1,j}}(\tilde{\mu}_X))$ is bounded, it suffices to exhibit a path from $\tilde{\mu}_X$ to a point in $Q(\alpha_{1,j})$ which makes only bounded progress in subsurfaces which interlock $\alpha_{1,j}$.  The augmented hierarchy path $\Gamma_{X, \tau}$ is precisely such a path.  Recall that $\mathcal{A}$ consists of all the $\widehat{K}$-large links between $\tilde{\mu}_X$ and $\tilde{\mu}_{\tau}$, which we have ordered by their appearance along $\Gamma_{X, \tau}$, and that $\alpha_1$ is the first curve in $\mathcal{A}$ to appear as a base curve along $\Gamma_{X,\tau}$.  Since Lemma \ref{r:sym fam time order} implies that the orbits in $\mathcal{A}$ are time-ordered together, it follows that any other curve $\beta \in \mathcal{A}$ which intersects $\alpha_{1,j}$ can only appear as a base curve along $\Gamma_{X,\tau}$ \emph{after} all progress through $\alpha_{1,j}$ has already been made.  By Lemma \ref{r:active interval}, $\Gamma_{X,\tau}$ makes a bounded amount of progress in subsurfaces which interlock $\alpha_{1,j}$ between $\tilde{\mu}_X$ and the first point along $\Gamma_{X,\tau}$ at which $\alpha_{1,j}$ appears in its base.\\ 

Thus $d_{\AM(S)}(\tilde{\mu}_X, \phi_{\alpha_{1,j}}(\tilde{\mu}_X)) \asymp 1$ for all $j$.\\

Since the $\phi_{\alpha_{1,j}}$ are Lipschitz (Lemma \ref{r:phi lipschitz}), it follows that
\[d_{\AM(S)}\left(\tilde{\mu}_X, \phi_{\alpha_{1,1}}(\tilde{\mu}_X)\right) \asymp d_{\AM(S)}\left(\phi_{\alpha_{1,2}}(\tilde{\mu}_X),\phi_{\alpha_{1,2}}(\phi_{\alpha_{1,1}}(\tilde{\mu}_X))\right) \asymp d_{\AM(S)}\left(\phi_{\alpha_{1,2}}(\tilde{\mu}_X),\phi_{\alpha_{1,1}\cup \alpha_{1,2}}(\tilde{\mu}_X)\right)\]
with the second coarse equality following from Proposition \ref{r:phi add}.\\

Since $d_{\AM(S)}\left(\tilde{\mu}_X, \phi_{\alpha_{1,2}}(\tilde{\mu}_X)\right)\asymp 1$, it follows from applying the triangle inequality that 
\[d_{\AM(S)}\left(\tilde{\mu}_X, \phi_{\alpha_{1,1}\cup {\alpha_{1,2}}}(\tilde{\mu}_X)\right) \asymp 1\]

Applying this observation a uniformly bounded number of times (for $n_1 \leq |H|$), we obtain $d_{\AM(S)}(\tilde{\mu}_X, \tilde{\mu}_{\alpha_1}) \asymp 1$, proving Claim 1.\\

Let $T_{\alpha_1} = \prod_{j=1}^{n_1}T_{\alpha_{1,j}}^{(-1)^{s_1} \cdot d_1}$, where $T_{\alpha_{1,j}}$ is the Dehn (half)twist around $\alpha_{1,j}$, $d_1 = d_{\alpha_1,1}(\tilde{\mu}_X,\tilde{\mu}_{\tau})$, and the sign  $s_1$ depends on whether $\pi_{\alpha_{1,1}}(\tilde{\mu}_X)$ differs from $\pi_{\alpha_{1,1}}(\tilde{\mu}_{\tau})$ by right or left Dehn (half) twists around $\alpha_{1,1}$.  Set $X_1 = T_{\alpha_1}(X)$ and let $\tilde{\mu}_{X_1}$ be its shortest augmented marking.\\

First, note that since $T_{\alpha_1} \in C_{\MCG(S)}(H)$ centralizes $H$ in $\MCG(S)$ and is thus contained in the normalizer, which stabilizes $\Fix(H)$, we have $X_1 \in \Fix(H)$.  Moreover, we claim that the distance between $X$ and $X_1$ is coarsely determined by the distance traveled in $\mathcal{A}_1$:
\begin{equation}
d_{\TT(S)}(X,X_1) \asymp \sum_{\alpha \in \mathcal{A}_1} \left[ \log d_{\alpha}(\tilde{\mu}_X, \tilde{\mu}_{X_1})\right]_{K_1} \label{r:Dehn close 1}
\end{equation}
and
\begin{equation}
\sum_{Y \subset S \setminus \mathcal{A}_1} \left[ d_{Y}(\tilde{\mu}_X, \tilde{\mu}_{X_1})\right]_{K_1}=0 \label{r:Dehn close 2}
\end{equation}

where $K_1$ is a constant depending only on $R$ and $S$.\\

Recall that Lemma \ref{r:sym sub proj} implies that $d_{\alpha_{1,i}}(\tilde{\mu}_X,\tilde{\mu}_{\tau}) \asymp_R d_{\alpha_{1,j}}(\tilde{\mu}_X,\tilde{\mu}_{\tau})$ for any $i,j$ and since $X$ is fixed and $\tau$ is has a bounded diameter orbit, it follows that $\pi_{\alpha_{1,i}}(\tilde{\mu}_X)$ differs from $\pi_{\alpha_{1,i}}(\tilde{\mu}_{\tau})$ by coarsely the same number of right or left Dehn (half)twists for all $i$, where the handedness is independent of $i$.  We immediately obtain $d_{\alpha_{1,i}}(\tilde{\mu}_{X_1},\tilde{\mu}_{\tau}) \asymp_R 1$ for all $i$.  Thus once we prove that (\ref{r:Dehn close 1}) and (\ref{r:Dehn close 2}) are true, it will follow from the triangle inequality that 
\begin{equation}
d_T(X_1, \tau) \asymp \sum_{\alpha \in \mathcal{A} \setminus \mathcal{A}_1} \left[\log d_{\alpha}(\tilde{\mu}_{X},\tilde{\mu}_{\tau})\right]_{K_1}\label{r:elim 1}
\end{equation}
and
\begin{equation}
\sum_{Y \subset S \setminus( \mathcal{A} \setminus \mathcal{A}_1)} \left[d_{Y}(\tilde{\mu}_{X_1},\tilde{\mu}_{\tau})\right]_{K_1} = 0 \label{r:elim 2}
\end{equation}

By establishing (\ref{r:elim 1}) and (\ref{r:elim 2}), we will have shown that $X_1$ has removed the curves in $\mathcal{A}_1$ as combinatorial obstacles between $X$ and $\tau$, while all other projections remain coarsely unchanged.  These equations are rephrased as the inductive hypothesis in (\ref{r:eq d to X}) and (\ref{r:eq d to tau}) below.\\

To see (\ref{r:Dehn close 1}) and (\ref{r:Dehn close 2}), observe that $\phi_{\mathcal{A}_1}(\tilde{\mu}_X), \tilde{\mu}_{X_1} \in Q(\mathcal{A}_1)$.  By Lemma \ref{r:distance prod}, the distance between $\phi_{\mathcal{A}_1}(\tilde{\mu}_X)$ and $\tilde{\mu}_{X_1}$ is coarsely determined by projections to subsurfaces $Y \subset \sigma(\mathcal{A}_1)$, with all other subsurface projections being uniformly bounded.  However, note that since $\phi_{\mathcal{A}_1}(\tilde{\mu}_X), \tilde{\mu}_{X_1} \in Q(\mathcal{A}_1)$, all base and transverse curves in $\phi_{\mathcal{A}_1}(\tilde{\mu}_X)$ and $\tilde{\mu}_{X_1}$ are disjoint from $\mathcal{A}_1$, and so $T_{\alpha_1}$ only acts nontrivially on the $\mathcal{A}_1$ coordinates of $\phi_{\mathcal{A}_1}(\tilde{\mu}_X)$ and $\tilde{\mu}_{X_1}$.  Lemma \ref{r:lip for mark proj} implies that $d_Y(\phi_{\mathcal{A}_1}(\tilde{\mu}_X), \tilde{\mu}_{X_1}) \asymp 1$ for all $Y \subset \sigma(\mathcal{A}_1)\setminus \mathcal{A}_1$, from which (\ref{r:Dehn close 1}) and (\ref{r:Dehn close 2}), and thus (\ref{r:elim 1}) and (\ref{r:elim 2}), follow for some choice of $K_1$ depending only on $R$ and $S$.\\

In summary, we have produced a point $X_1\in \Fix(H)$ whose distance to $\tau$ is determined by one less set of annuli, while the distances of projections to all other subsurfaces are coarsely unchanged.\\

We remark that in the above calculations, we repeatedly made coarse estimates to determine that the distance in (\ref{r:Dehn close 2}) is bounded.  Since we did so only finitely many times, where the number of times depended only on the topology of $S$ and the almost-fixed constant $R$, it follows that the coarseness of our estimates is still uniformly bounded as a function of $R$ and $S$.\\

In what follows, we make an inductive argument in which we perform a similar series of computations to create the sequence of fixed points $X_1, \dots, X_N$.  With the last point, $X_N$, we will have moved past each of the families in $\mathcal{A}$, at each step leaving all complementary subsurface projections coarsely fixed.  Since $N$ was a number which depended only on $R$ and $S$, we find a bound for $d_T(X_N, \tau)$ that depends only on $R$ and $S$.  Since $R$ was a fixed constant independent of $\tau$, it follows that $d_T(X_N, \tau)$ and thus $d_T(\tau, \Fix(H))$ are uniformly bounded in terms of $R$ and $S$, completing the proof.\\

We proceed by induction on the $A_i$.  Suppose we have created a sequence of fixed points, $X_1, \dots, X_i \in \Fix(H)$ with shortest augmented markings $\tilde{\mu}_{X_1}, \dots, \tilde{\mu}_{X_i}$ and a sequence of constants, $K_i$ depending only on $R$ and $S$, such that for each $j \leq i$ the following properties hold:

\begin{enumerate}
\item For every subsurface $Y \subset S$ which is not an annulus with core curve $\alpha_{l,m} \in \mathcal{A}$ for $l \leq j$, we have $d_Y(\tilde{\mu}_X,\tilde{\mu}_{X_j}) < K_j$ \label{r:eq d to X}
\item For every subsurface $Y \subset S$ which is not an annulus with core curve $\alpha_{l,m} \in \mathcal{A}$ for $l \geq j$, we have $d_Y(\tilde{\mu}_{X_j}, \tilde{\mu}_{\tau}) < K_j$ \label{r:eq d to tau}
\end{enumerate}

We have already shown that the base case of $i=1$ holds above in (\ref{r:elim 1}) and (\ref{r:elim 2}). \\

Note that (\ref{r:eq d to X}) and the triangle inequality imply that $d_{\alpha_{l,m}}\left(\tilde{\mu}_{X_j},\tilde{\mu}_{X}\right)\asymp_R 1$ for all $j\geq i$, $l \geq j$, and $m \leq n_j$.  Similarly, (\ref{r:eq d to tau}) and the triangle inequality imply that $d_{\alpha_{l,m}}\left(\tilde{\mu}_{X_j},\tilde{\mu}_{\tau}\right)\asymp_R 1$ for all $j\leq i$, $l \leq j$, and $m \leq n_j$.\\

Since $\mathcal{A}$ consisted of $N$ orbits of curves with $N=N(R,S)>0$, once the inductive step is proven, we will have constructed a fixed point $X_N \in \Fix(H)$ which satisfies the inequality in (\ref{r:eq d to tau}).  Since $j$ in (\ref{r:eq d to tau}) is bounded by $N$, it will follow that $d_T(X_N, \tau) < K_N$, where $K_N$ depends only on $R$ and $S$, completing the proof.\\

We now proceed to prove the inductive step.  The construction of $X_{i+1}$ from $X_i$ is similar to the construction of $X_1$ from $X$, but there are now are more quantities to manage.  Let $\tilde{\mu}_{i+1} = \phi_{\mathcal{A}_{i+1}}(\tilde{\mu}_{X_i})$.  As before, we begin with the following claim:\\

\underline{Claim $(i+1)$}: $d_{\AM(S)}\left(\tilde{\mu}_{X_i}, \tilde{\mu}_{i+1}\right) \asymp 1$.\\

As with Claim 1, the proof of Claim $(i+1)$ involves showing that $d_{\AM(S)}\left(\tilde{\mu}_{X_i}, \tilde{\mu}_{i+1}\right) \asymp 1$ for $1\leq j \leq n_{i+1}$ and then repeatedly applying Lemma \ref{r:phi lipschitz} and Proposition \ref{r:phi add} and the triangle inequality.\\

Let $1\leq j\leq n_{i+1}$.  By Lemma \ref{r:cpj}, $\phi_{\alpha_{i+1,j}}$ is coarsely the closest point projection to $Q(\alpha_{i+1,j})$, so Lemma \ref{r:distance to q} implies that
\[d_{\AM(S)}(\tilde{\mu}_{X_i}, \tilde{\mu}_{i+1}) \asymp \sum_{Y \pitchfork \alpha_{i+1,j}} \left[d_Y(\tilde{\mu}_{X_i}, \alpha_{i+1,j})\right]_{L_1}\]
and
\[\sum_{Y \subset S \setminus \alpha_{i+1,j}} \left[d_Y(\tilde{\mu}_{X_i}, \alpha_{i+1,j})\right]_{L_1}= 0\]

where $L_1$ is the uniform constant from Lemma \ref{r:cpj}.\\

Let $\tilde{\mu}_{\alpha_{i+1,j}} \in \Gamma_{X,\tau}$ be the first point along $\Gamma_{X, \tau}$ in which $\alpha_{i+1,j}$ appears as a base curve.  If $\alpha_{l,m} \in \mathcal{A}$ is such that $\alpha_{l,m} \pitchfork \alpha_{i+1,j}$ and $l \leq i$ and $m \leq n_l$, then Lemma \ref{r:sym fam time order} implies that $\alpha_{l,m} \prec_t \alpha_{i+1,j}$.  Lemma \ref{r:active interval} implies that the active segment of $\alpha_{l,m}$ entirely precedes the active segment of $\alpha_{i+1,j}$, of which $\tilde{\mu}_{\alpha_{i+1,j}}$ is the first point.  Thus $d_{\alpha_{l,m}}(\tilde{\mu}_{\alpha_{i+1,j}}, \tilde{\mu}_{\tau}) \asymp 1$ by Lemma \ref{r:active interval}.  Since $d_{\alpha_{l,m}}\left(\tilde{\mu}_{X_i},\tilde{\mu}_{\tau}\right)\asymp 1$ by inductive assumption (\ref{r:eq d to tau}), the triangle inequality implies that $d_{\alpha_{l,m}}\left(\tilde{\mu}_{X_i},\tilde{\mu}_{\alpha_{i+1,j}}\right)\asymp 1$ for all $l \leq i, m \leq n_l$ for which $\alpha_{l,m} \pitchfork \alpha_{i+1,j}$.  Since $\alpha_{i+1,j} \in \mathrm{base}(\tilde{\mu}_{\alpha_{i+1,j}})$, it follows that $d_Y(\alpha_{i+1,j} ,\tilde{\mu}_{\alpha_{i+1,j}}) \asymp 1$ for any $Y \pitchfork \alpha_{i+1,j}$ and thus $d_{\alpha_{l,m}}\left(\tilde{\mu}_{X_i},\alpha_{i+1,j}\right)\asymp 1$. \\

On the other hand, for any $\beta \in \mathcal{A}$ with $\alpha_{i+1,j} \prec_t \beta$, Lemma \ref{r:active interval} implies that $d_{\beta}(\tilde{\mu}_{X}, \tilde{\mu}_{\alpha_{i+1,j}}) \asymp 1$.  Thus inductive assumption (\ref{r:eq d to X}) and the triangle inequality imply that $d_{\beta}(\tilde{\mu}_{X_i}, \tilde{\mu}_{\alpha_{i+1,j}}) \asymp 1$ for any such $\beta$.\\

To summarize, we have shown:

\begin{align*}
d_{\AM(S)}\left(\tilde{\mu}_{X_i},\phi_{\alpha_{i+1, j}}(\tilde{\mu}_{X_i}))\right) &\asymp \sum_{Y \pitchfork \alpha_{i+1,j}} \left[d_Y(\tilde{\mu}_{X_i}, \phi_{\alpha_{i+1, j}}(\tilde{\mu}_{X_i}))\right]_{K'}\\
&\asymp \sum_{Y \pitchfork \alpha_{i+1,j}} \left[d_Y(\tilde{\mu}_{X_i}, \tilde{\mu}_{\alpha_{i+1,j}})\right]_{K'}\\
&\asymp \sum_{\overset{\alpha_{l,m} \pitchfork \alpha_{i+1,j}}{l\leq i}} \left[d_{\alpha_{l,m}}(\tilde{\mu}_{X_i}, \tilde{\mu}_{\alpha_{i+1,j}})\right]_{K'}\\
&\asymp 1
\end{align*}
where $K' = \max\{K_i,L_1\}$, which we note depends only on $R$ and $S$.\\ 

Claim $(i+1)$ follows by applying Lemma \ref{r:phi lipschitz} and Proposition \ref{r:phi add} a uniformly bounded number of times, as in the proof of Claim 1.\\

We now proceed to create $X_{i+1}$ from $X_i$ as we did $X_1$ from $X_0$.  Let $T_{\alpha_{i+1}} = \prod_{j=1}^{n_{i+1}}T_{\alpha_{i+1,j}}^{(-1)^{s_{i+1}}d_{i+1}}$, where $T_{\alpha_{i+1,j}}$ is the Dehn (half)twist around $\alpha_{i+1,j}$, $d_{i+1} = d_{\alpha_{i+1,1}}(\tilde{\mu}_X,\tilde{\mu}_{\tau})$, and the sign  $s_{i+1}$ depends on whether $\pi_{\alpha_{i+1,1}}(\tilde{\mu}_X)$ differs from $\pi_{\alpha_{i+1,1}}(\tilde{\mu}_{\tau})$ by right or left Dehn (half) twists around $\alpha_{i+1,1}$.  Set $X_{i+1} = T_{\alpha_{i+1}}(X_i)$ and let $\tilde{\mu}_{X_{1+1}}$ be its shortest augmented marking.\\

Once again $T_{\alpha_{i+1}} \in C_{\MCG(S)}(H)$ centralizes $H$, so it stabilizes $\Fix(H)$ and $X_{i+1}\in \Fix(H)$.  We claim that $X_{i+1}$ satisfies the properties in the inductive assumptions (\ref{r:eq d to X}) and (\ref{r:eq d to tau}) above.\\

Lemma \ref{r:sym sub proj} implies that $d_{\alpha_{i+1,j}}(\tilde{\mu}_X,\tilde{\mu}_{\tau}) \asymp_R d_{\alpha_{i+1,l}}(\tilde{\mu}_X,\tilde{\mu}_{\tau})$ for any $j,l$, and since $X$ is fixed and $\tau$ is has a bounded diameter orbit, it follows that $\pi_{\alpha_{i+1,j}}(\tilde{\mu}_X)$ differs from $\pi_{\alpha_{i+1,j}}(\tilde{\mu}_{\tau})$ by coarsely the same number of right or left Dehn (half)twists for all $j$, where the handedness is independent of $i$.  It follows immediately that $d_{\alpha_{i+1,j}}(\tilde{\mu}_{X_{i+1}},\tilde{\mu}_{\tau}) \asymp_R 1$ for all $j$.\\

Observe that $\phi_{\mathcal{A}_{i+1}}(\tilde{\mu}_{X_i}), \tilde{\mu}_{X_{i+1}} \in Q(\mathcal{A}_{i+1})$.  By Lemma \ref{r:distance prod}, the distance between $\phi_{\mathcal{A}_{i+1}}(\tilde{\mu}_{X_i})$ and $\tilde{\mu}_{X_{i+1}}$ is coarsely determined by projections to subsurfaces $Y \subset \sigma(\mathcal{A}_{i+1})$, with all other subsurface projections being uniformly bounded.  However, note that since $\phi_{\mathcal{A}_{i+1}}(\tilde{\mu}_{X_i}), \tilde{\mu}_{X_{i+1}} \in Q(\mathcal{A}_{i+1})$, all base and transverse curves in $\phi_{\mathcal{A}_{i+1}}(\tilde{\mu}_{X_i})$ and $\tilde{\mu}_{X_{i+1}}$ are disjoint from $\mathcal{A}_{i+1}$, and so $T_{\alpha_{i+1}}$ only acts nontrivially on the $\mathcal{A}_{i+1}$ coordinates of $\phi_{\mathcal{A}_{i+1}}(\tilde{\mu}_{X_i})$ and $\tilde{\mu}_{X_{i+1}}$.  Thus $d_Y(\phi_{\mathcal{A}_{i+1}}(\tilde{\mu}_{X_i}), \tilde{\mu}_{X_{i+1}}) \asymp_R 1$ for all $Y \subset \sigma(\mathcal{A}_{i+1})\setminus \mathcal{A}_{i+1}$.  Equations (\ref{r:eq d to X}) and (\ref{r:eq d to tau}) for $j=i+1$ follow immediately from the triangle inequality and the inductive assumptions that (\ref{r:eq d to X}) and (\ref{r:eq d to tau}) hold for $X_i$.\\

This completes the inductive step and thus the proof.

\end{proof}

\section{Coarse barycenters for the Teichm\"uller metric}

The goal of this section is to prove the following theorem:

\begin{theorem}[Coarse fixed barycenters for $(\TT(S), d_T)$] \label{r:bary teich}
There are $ \widetilde{K}, \widetilde{C}>0$ such that for any $\sigma \in \TT(S)$ and any finite order $f \in \MCG(S)$, there is a fixed point $X \in Fix(\langle f \rangle)$ such that 
\[d_T(\sigma, X) <  \widetilde{K} \cdot d_{\TT(S)}(\sigma, f\cdot \sigma) + \widetilde{C}\] 
\end{theorem}

The proof relies in an essential way on Tao's main technical result \cite[Theorem 4.0.2]{Tao13}, from which the linearly bounded conjugator property for $\MCG(S)$ for finite order elements follows almost immediately.  She proves that there are coarse barycenters in $\MM(S)$ for finite order elements of $\MCG(S)$:

\begin{theorem} [Coarse barycenters for $\MM(S)$; Theorem 4.0.2 in \cite{Tao13}]\label{r:jing's theorem}
There are $R, K,C>0$ depending only on $S$, so that for any marking $\mu \in \MM(S)$ and finite order $f \in \MCG(S)$, there is a $\mu_0 \in \MM(S)$ with $\mathrm{diam}_{\MM(S)}(\langle f \rangle \cdot \mu_0) < R$, such that
\[d_{\MM(S)}(\mu,\mu_0) < K \cdot d_{\MM(S)}(\mu, f\cdot \mu) + C\]
\end{theorem}

The proof of Theorem \ref{r:jing's theorem} proceeds by choosing a marking in $\MM(S)$ with uniformly bounded $f$-orbit and then step by step reducing the complexity of the large subsurface projections between $\mu$ and the chosen marking, each step resulting in a new marking with uniformly bounded $f$-orbit, whose combinatorial relationship with $\mu$ is simpler.  At the core of the proof are two technical Propositions 4.2.1 and 4.2.2, which construct the new markings.  We need some observations about the proof these propositions.\\

Let $\mu_0\in \MM(S)$ have a uniformly bounded $f$-orbit and suppose that $Y \in \mathcal{L}_{\widehat{K}}(\mu,\mu_0)$ is a $\widehat{K}$-large link with $Y \varsubsetneqq S$ a proper subsurface, where $\widehat{K}$ is the constant from Tao's Lemma \ref{r:sym links} and $\mathcal{L}_{\widehat{K}}(\mu,\mu_0)$ is the set of $\widehat{K}$ large links between $\mu$ and $\mu_0$.  Let $L_Y \in \mathbb{N}$ be the smallest natural number such that $f^{L_Y+1}$ is the first return map of $f$ to $Y$ and set $\Lambda_Y = \partial Y \cup f\cdot \partial Y \cup \cdots \cup f^{L_Y} \partial Y$.  We note that Lemma \ref{r:sym links} implies that $f^i \cdot Y \in \mathcal{L}_{\widehat{K}}(\mu,\mu_0)$ for each $i$.\\

Proposition 4.2.2 of \cite{Tao13} produces a new marking $\mu_1 \in \MM(S)$ with uniformly bounded $f$-orbit with $\Lambda_Y \subset \mathrm{base}(\mu_1)$, such that $f^i \cdot Y \notin \mathcal{L}_{\widehat{K}}(\mu,\mu_1)$ for each $i$ and if moreover $Z \in \mathcal{L}_{\widehat{K}}(\mu,\mu_1)$ and $Z \notin \mathcal{L}_{\widehat{K}}(\mu,\mu_0)$, then $Z \subset Y$ is a proper subsurface and thus has lower complexity.  The marking $\mu_1$ is first constructed via marking projections.  Namely, one chooses correct transversals for the curves in $\Lambda_Y$, then builds markings in $\MM(f^i\cdot Y)$ for each $i$.  To complete these pieces to a marking on all of $S$, one induces the structure of $\mu$ on $S \setminus \left( \coprod_{1 \leq i \leq L_Y} f^i\cdot Y\right)$ by projecting $\mu_0$ to a marking on each component thereof.  In particular, this means that $\mu_0$ and $\mu_1$ have uniformly bounded projections to any subsurface of $S \setminus \left(\coprod_{1 \leq i \leq L_Y} f^i\cdot Y\right)$.\\

The proof of Theorem \ref{r:bary teich} proceeds by analyzing the short $f$-symmetric curves of the arbitrary point $\sigma \in \TT(S)$ and choosing an initial point, $X'$, whose length and twisting coordinates in these short curves are sufficiently close to those of $\sigma$.  We then apply Tao's Theorem \ref{r:jing's theorem} to the marking, $\mu_{X'}$, underlying a shortest augmented marking for $X'$.  By the above observations, the result is a new almost-fixed marking, $\mu_X$, whose base curves contain the short $f$-symmetric curves of $\sigma$ and whose transversals to these curves have changed a uniformly bounded amount compared to those of $\mu_{X'}$.  We may then build an almost-fixed augmented marking, $\tilde{\mu}_{X''}$, whose projections to the horoballs over the $f$-symmetric short curves of $\sigma$ are the same as those of $\tilde{\mu}_{X'}$.  After performing similar calculations to the proof of the Main Theorem \ref{r:main}, we find that any point $X'' \in \TT(S)$ whose shortest augmented marking is $\tilde{\mu}_{X''}$ is an $R$-almost-fixed barycenter for $\sigma$ in $\TT(S)$, for some $R$ depending only on $S$.  An application of the Main Theorem \ref{r:main} produces the desired fixed point, $X \in Fix(\langle f \rangle)$.

\begin{proof}[Proof of Theorem \ref{r:bary teich}]

Let $\sigma \in \TT(S)$ be arbitrary and $f\in \MCG(S)$ finite order.  Let $\epsilon_0>0$ be as in Definition \ref{r:epsilon} with $H = \langle f \rangle$.\\

Let $\Lambda_{\epsilon_0, sym}(\sigma) = \{\lambda | l_{\sigma}(f^k\cdot \lambda) < \epsilon_0, \forall k\}$, the set of $f$-symmetric short curves of $\sigma$.  We note that it is possible that other $f$-symmetric curves will be short in $\sigma$, but we are only interested in those whose entire $f$-orbit is short in $\sigma$.\\

Let $\tilde{\mu}_{\sigma} \in \AM(S)$ be a shortest augmented marking for $\sigma$.  Remark \ref{r:short base} implies that $\Lambda_{\epsilon_0, sym}(\sigma) \in \mathrm{base}(\tilde{\mu}_{\sigma})$.  Decompose the curves in $\Lambda_{\epsilon_0, sym}(\sigma)$ into their $f$-orbits, $\Lambda_1, \dots, \Lambda_k$.  For each $i$, there is an $N_i>0$ so that $\Lambda_i = \{\lambda_i, f \cdot \lambda_i, \cdots, f^{N_i} \cdot \lambda_i\}$.  For each $1 \leq j \leq N_i$, let $b_{i,j}\in \HHH_{\Lambda_i}$ be a coarse barycenter of $\pi_{\HHH_{f^j \cdot \lambda_i}}(\langle f \rangle \cdot \sigma)$, the projection of the $f$-orbit of $\tilde{\mu}_{\sigma}$ to $\HHH_{f^j \cdot \lambda_i}$; that is, $d_{\HHH_{f^j \cdot \lambda_i}}(f^k \cdot \tilde{\mu}_{\sigma}, b_{i,j}) \prec d_{\TT(S)} (\sigma, f \cdot \sigma)$ for each $1 \leq k \leq N_i$.\\

Let $X' \in \TT(S)$ be any point whose shortest augmented marking $\tilde{\mu}_{X'} \in \AM(S)$ satisfies the following conditions:

\begin{enumerate}
\item $\mathrm{base}(\tilde{\mu}_{X'}) = \mathrm{base}(\tilde{\mu}_{\sigma})$
\item For all $\alpha \in \mathrm{base}(\tilde{\mu}_{\sigma}) \setminus \Lambda_{\epsilon_0, sym}(\sigma)$, $d_{\HHH_{\alpha}}(\tilde{\mu}_{X'}, \tilde{\mu}_{\sigma}) \asymp 1$
\item For each $i$ and $1 \leq j \leq N_i$, $d_{\HHH_{f^j \cdot \lambda_i}}(\tilde{\mu}_{X'}, b_{i,j}) \asymp 1$
\end{enumerate}

Observe that both $\tilde{\mu}_{\sigma}, \tilde{\mu}_{X'} \in Q(\Lambda_{\epsilon_0, sym}(\sigma))$.  By the choice of $X'$, all projections of $\tilde{\mu}_{\sigma}$ and $\tilde{\mu}_{X'}$ to horoballs over curves in $\Lambda_{\epsilon, sym}(\sigma)$ are linearly bounded in terms of $d_{\TT(S)}(\sigma, f\cdot \sigma)$: that is, there exist $K',C'>0$ depending only on $S$ such that $d_{\HHH_{\lambda}}(\tilde{\mu}_{X'}, \tilde{\mu}_{\sigma}) \leq K' \cdot d_{\TT(S)}(\sigma, f\cdot \sigma) + C'$ for all $\lambda \in \Lambda_{\epsilon_0, sym}(\sigma)$.  Moreover, since $\tilde{\mu}_{\sigma}, \tilde{\mu}_{X'} \in Q(\Lambda_{\epsilon_0, sym}(\sigma))$, Lemma \ref{r:distance to q}  implies that for any other subsurface $Y$ for which $d_Y(\tilde{\mu}_{\sigma}, \tilde{\mu}_{X'})$ is sufficiently large, we must have $Y \subset S \setminus \Lambda_{\epsilon, sym}(\sigma)$.\\

Let $\widehat{K}>0$ be the constant from Tao's Lemma \ref{r:sym links} with $H = \langle f \rangle$ and let $\mathcal{L}_{\widehat{K}}(\tilde{\mu}_{\sigma}, \tilde{\mu}_{X'})$ be the collection of $\widehat{K}$-large links between $\tilde{\mu}_{\sigma}$ and $\tilde{\mu}_{X'}$.  As noted at the end of the previous paragraph, each $Y \in \mathcal{L}_{\widehat{K}}(\tilde{\mu}_{\sigma}, \tilde{\mu}_{X'})$ satisfies $Y \subset S \setminus \Lambda_{\epsilon_0, sym}(\sigma)$.\\

Let $\mu_{X'} \in \MM(S)$ be the marking underlying $\tilde{\mu}_{X'}$.  We now apply Tao's Theorem \ref{r:jing's theorem} to $\mu_{X'}$.  By the discussion of the proof of \cite{Tao13}[Proposition 4.2.2], Tao's Theorem \ref{r:jing's theorem} produces an $R$-almost fixed marking $\mu_{X''} \in \MM(S)$, which has the property that, for each $Y \subset S$, $d_Y(\mu_{X''}, \mu_{\sigma}) < K' \cdot d_Y(\mu_{\sigma}, f\cdot \mu_{\sigma}) + C'$, where $K', C'>0$ depend only on $S$.  Moreover, we have that $\Lambda_{\epsilon_0, sym}(\sigma) \subset \mathrm{base}(\mu_{X''})$, so we may build an augmented marking $\tilde{\mu}_{X''} \in \AM(S)$ whose length coordinates for the curves in $\Lambda_{\epsilon_0, sym}(\sigma)$ are those of $\tilde{\mu}_{X'}$.\\

We have already shown that $d_Y(\tilde{\mu}_{X''}, \tilde{\mu}_{\sigma}) < K'' \cdot d_Y(\tilde{\mu}_{\sigma}, f\cdot \tilde{\mu}_{\sigma})+C''$ for any subsurface $Y \subset S$ (including annuli), where $K'', C''>0$ depend only on $S$.  It remains to show that we have a similar bound on projections to all horoballs.\\

By construction, we have such a bound on any projection to a horoball over one of the curves in $\Lambda_{\epsilon_0, sym}(\sigma)$.  If $\lambda \in \CC(S)$ and $\lambda \notin \Lambda_{\epsilon_0, sym}(\sigma)$, then it follows that at least one curve $f^i \cdot \lambda$ in the $f$-orbit of $\lambda$ satisfies $l_{\sigma}(f^i \cdot \lambda) > \epsilon_0$.  In particular, for such a curve $f^i \cdot \lambda$, the projection $\pi_{\HHH_{f^i \cdot \lambda}}(\tilde{\mu}_{\sigma})$ must bounded coarse length coordinate equal to $0$; the coarse length coordinate of $\pi_{\HHH_{f^i \cdot \lambda}}(\tilde{\mu}_{X''})$ is $0$ by construction.  As the twisting coordinate of $\pi_{\HHH_{f^i \cdot \lambda}}(\tilde{\mu}_{\sigma})$ and $\pi_{\HHH_{f^i \cdot \lambda}}(\tilde{\mu}_{X''})$ satisfy the above desired bound, it follows that $d_{\HHH_{f^i \cdot \lambda}}(\tilde{\mu}_{X''}, \tilde{\mu}_{\sigma}) < K' \cdot d_{\HHH_{f^i \cdot \lambda}}(\tilde{\mu}_{\sigma}, f\cdot \tilde{\mu}_{\sigma})+C'$.\\

Let $X'' \in \TT(S)$ be any point whose shortest augmented marking is $\tilde{\mu}_{X''}$.  Then there are $K''', C'''>0$ depending only on $S$ such that

\[d_{\TT(S)}(X'', \sigma) < K''' \cdot d_{\TT(S)}(\sigma, f\cdot \sigma) + C'''\]

Applying the Main Theorem \ref{r:main}, it follows that there are $\widetilde{K}, \widetilde{C}>0$ depending only on $S$ and a fixed point $X \in \Fix(\langle f\rangle)\subset \TT(S)$ such that

\[d_{\TT(S)}(X, \sigma) < \widetilde{K} \cdot d_{\TT(S)}(\sigma, f\cdot \sigma) + \widetilde{C}\]

as desired.
\end{proof}

\begin{remark}[Theorem \ref{r:bary teich} for arbitrary finite subgroups]
We expect that Theorem \ref{r:bary teich} can be generalized to hold for any finite subgroup $H \leq \MCG(S)$.  This might be accomplished by generalizing Tao's Theorem \ref{r:jing's theorem}, but this would require a nearly complete reworking of her proof.  
\end{remark}

\begin{remark}[Independence of Theorem \ref{r:main} and Theorem \ref{r:jing's theorem}]\label{r:indep}
At first glance, it may seem that one might derive Theorem \ref{r:main} from Theorem \ref{r:jing's theorem} or vice versa.  The former does not imply the latter, since the bound in Theorem \ref{r:main} from the starting point to $\Fix(H)$ is not linear in terms of the diameter of the orbit of the starting point.  On the other hand, the latter does not imply the former, for it can at best produce an almost fixed point, when a genuine fixed point is needed.  What is more, Theorem \ref{r:main} holds for any finite subgroup of $\MCG(S)$ and Theorem \ref{r:jing's theorem} is only known for finite order elements.
\end{remark}

\section{Non-quasiconvexity of $\Fix^T_R(H)$}

This purpose of this section is to prove the following theorem:

\begin{theorem}\label{r:non qc}
There exist an $R>0$, a surface $S$, and a finite subgroup $H \subset \MCG(S)$ such that $\Fix^T_R(H)$ is not quasiconvex.
\end{theorem}

The example built in Theorem \ref{r:non qc} is based on Rafi's example in \cite{Raf14}[Theorem 7.3] of two Teichm\"uller geodesic segments which start and end at a bounded distance from each other and yet do not fellow travel.  These two geodesics segments necessarily live in a thin part of $\TT(S)$, as \cite{Raf14}[Theorem 7.1] proves that this phenomenon does not occur when the endpoints are thick.  Our construction requires the techniques from Rafi's example, so we present Rafi's construction at the beginning of our proof.  After the proof of the theorem, we remark on how this theorem could be generalized.  We expect that $\Fix^T_R(H)$ is typically not quasiconvex.

\begin{proof}[Proof of Theorem \ref{r:non qc}]

Fix $d>0$.  Let $S_0$ be the closed genus 2 surface and let $\gamma\in \CC(S_0)$ be a separating curve on $S_0$.  Let $Y, Z \subset S_0$ be the two once-punctured tori which are the complements of $\gamma$.  In his construction, Rafi builds two Teichm\"uller geodesics $\mathcal{G}_1, \mathcal{G}_2: [0,2d]\rightarrow \TT(S_0)$ such that $ d_T(\mathcal{G}_1(0),\mathcal{G}_2(0)) \asymp 1$ and $d_T(\mathcal{G}_1(2d), \mathcal{G}_2(2d)) \asymp 1$, but $d_T(\GG_1(d), \GG_2(d)) \overset{\cdot}{\succ} d$, where $d>0$ can be chosen to be as large as necessary.  We now sketch how this works.\\

First, choose an Anosov map $\phi$ on a torus and choose a flat torus $T$ on the axis of $\phi$ so that the vertical direction in $T$ coincides with the unstable foliation of $\phi$.  Cut open a slit in $T$ of size $\rho= c \cdot e^{-\frac{d}{2}}$ and angle $\frac{\pi}{4}$ ($0<c<1$ is determined later).  Now fix a marking homeomorphism from $Y$ to the slit torus and label this marked surface $T_0$.  Set 
\begin{equation*}
T_{t} = \begin{bmatrix}
e^t & 0 \\
0&e^{-t} \end{bmatrix} T_{0}
\end{equation*}

Each $T_t$ is still a marked surface.  The length of the slit on $T_t$ is minimized at $t=0$ and grows exponentially as $t \rightarrow \pm \infty$.  The idea is that for $-\frac{d}{2} \leq t \leq \frac{d}{2}$ and $c$ small enough, the slit is short and every curve on $T_t$ has length comparable with $1$.\\

One can then scale down a copy of $T_{-\frac{d}{2}}$ by a factor of $\delta << \epsilon$; call this copy $\delta T_{-\frac{d}{2}}$.  Now cut a new slit in a different copy of $T$ with the same length as $\delta T_{-\frac{d}{2}}$ and fix a homeomorphism from $Z$ to this new slitted torus $T'$.  Gluing $T'$ to $\delta T_{-\frac{d}{2}}$ along their slits, one can build a quadratic differential and define a Teichm\"uller geodesic $\mathcal{G}_1: [0,2d] \rightarrow \TT(S_0)$.\\

One then builds a new Teichmuller geodesic by splicing together $T'$ with $\delta T_{-\frac{3d}{2}}$, called $\mathcal{G}_2: [0,2d] \rightarrow \TT(S_0)$.  The main idea of the construction is that while the endpoints of $\mathcal{G}_1$ and $\mathcal{G}_2$ are close (this follows quickly from the details of the construction), $\mathcal{G}_1$ moves through $Y$ during $[0,d]$ and $\mathcal{G}_2$ moves through $Y$ during $[d,2d]$, both at a linear rate, so that at time $t=d$, their projections to $Y$ are approximately $d$ apart, $d_{Y}(\mathcal{G}_1(d), \mathcal{G}_2(d)) \asymp d$.\\

We now build our examples.  We first lift $\GG_1$ and $\GG_2$ to $\TT(S)$, where $S$ is an appropriate finite cover with deck group isomorphic to $H \leq \MCG(S)$, where the lifts of $Y$ and $Z$ fill $S$, and the lifts $\GG'_1$ and $\GG'_2$ are now geodesics between points in $\Fix(H)$.  Then, using Rafi's construction, we build a new geodesic $\GG: [0,2d]\rightarrow  \TT(S)$ which starts and ends at almost-fixed points (with the almost-fixed constant to be determined below), and which performs the restriction of $\GG_1$ to $Y$ on one of the lifts of $Y$ and the restriction of $\GG_2$ to $Y$ on the other lifts of $Y$.  Consequently, the projections of $\tilde{\mu}_{\GG(d)}$ to the various lifts of $Y$ disagree by a factor of at least $d$, and so it follows that $\GG(d)$ cannot be close to $\Fix(H)$.\\

Let $S$ be the degree 4 cover of $S_0$ pictured in Figure \ref{fig:cexcover}.  Note that each of $Y$ and $Z$ lifts to two disjoint subsurfaces of $S$, say $Y_1, Y_2$ and $Z_1,Z_2$, with $S = Y_1 \cup Y_2 \cup Z_1 \cup Z_2$.  Let $QD(S_0)$ denote the space of holomorphic quadratic differentials on $S_0$.  Let $q_0, \widebar{q}_0 \in QD(S_0)$ be the quadratic differentials which define the geodesic segments $\GG_1$ and $\GG_2$ in $\TT(S_0)$, which were glued together from quadratic differentials on $Y$ and $Z$, $q_Y\in QD(Y)$ and $q_Z \in QD(Z)$.  The quadratic differentials $q_0$ and $\widebar{q}_0$ lift to pairs of quadratic differentials $q'_0, \widebar{q}'_0 \in QD(S)$.  Similarly, $q_Y$ and $q_Z$ lift to quadratic differentials $q'_{Y_1} \in QD(Y_1), q'_{Y_2} \in QD(Y_2)$ and $q'_{Z_1} \in QD(Z_1), q'_{Z_2} \in QD(Z_2)$, respectively.  These are the building blocks of our desired geodesic in $\TT(S)$.\\

\begin{figure} 
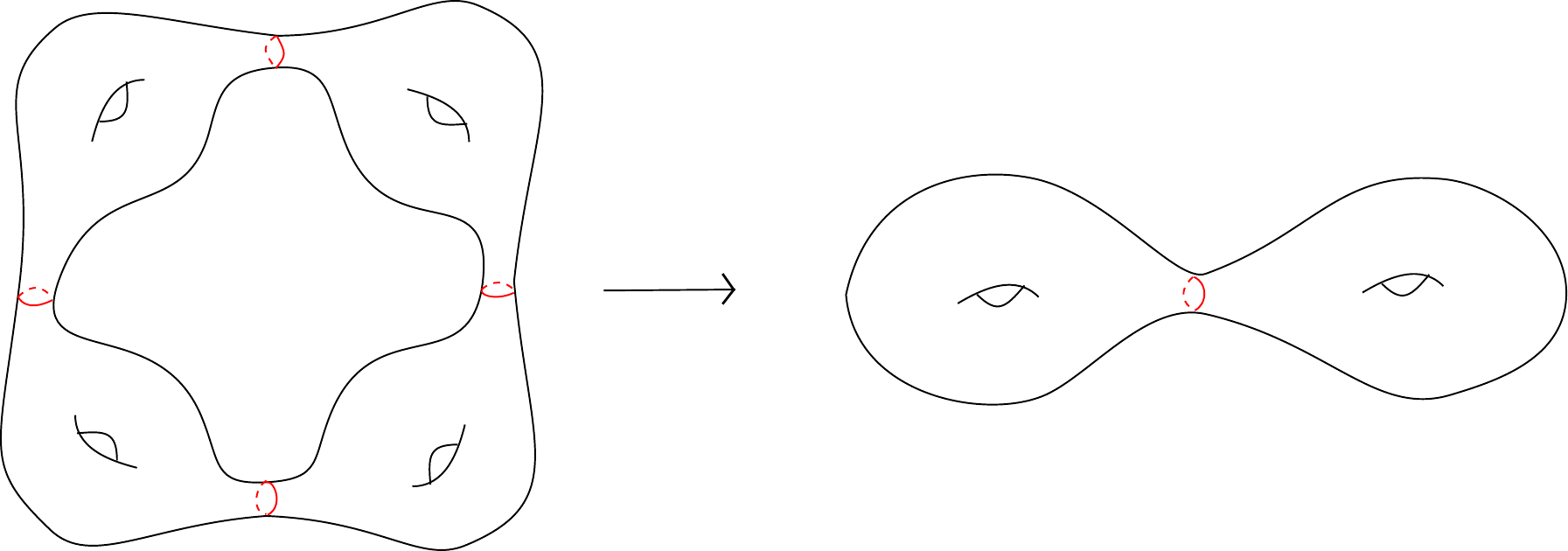
\caption{A genus five, degree-4 cover $S$ of the genus two surface $S_0$.  In $S_0$, the \textcolor{red}{red} curve $\gamma$ separates $Y$ and $Z$, each of which lift to two subsurfaces (clockwise from the upper-left) $Y_1, Y_2, Z_1,$ and $Z_2$, which are separated by the lifts of $\gamma$ and comprise $S$.}
\label{fig:cexcover}
\end{figure}

We now closely follow Rafi's construction.  Let $\phi$ be the same Anosov map on a torus and let $T$ be the same flat structure thereon used to create $\GG_1$ and $\GG_2$.  Recall that $T$ was chosen so that the vertical direction on $T$ matches the unstable foliation of $\phi$.  Instead of cutting one slit in $T$, cut open two parallel but not colinear slits in $T$ of size $\rho = c e^{-\frac{d}{2}}$ and of angle $\frac{\pi}{4}$, where the constant $0<c<1$ is specified shortly.  Fix a homeomorphism from $Y_1$ to this double-slit torus and called this marked flat surface $T_{Y_1,0}$.  Set
\begin{equation*}
T_{Y_1,t} = \begin{bmatrix}
e^t & 0 \\
0&e^{-t} \end{bmatrix} T_{Y_1,0}
\end{equation*}

For any $t$, $T_{Y_1,t}$ is still a marked surface and the slits have minimum length at $t=0$, growing exponentially as $t\rightarrow \pm \infty$.  For $-\frac{d}{2} \leq t\leq \frac{d}{2}$, the length of the slits is smaller than $c$, but since the stable and unstable foliations of $q_{Y_1}$ are cobounded, the length of any curve in $Y_{1,t}$ is comparable with 1.  As with Rafi's example, when $c$ is sufficiently small, $T_{Y_1,t}$ is an isolated subsurface when we glue it with the other slit tori to form $\GG$.  Choose $\delta \ll \rho$ as does Rafi, and let $q$ be the quadratic differential defined by gluing $T$ to $\delta T_{Y_1,-\frac{d}{2}}$ to another copy of $T$ to $\delta T_{Y_2, -\frac{3d}{2}}$ back onto the first copy of $T'$.\\

The details of this gluing are as follows.  We first scale down the given slitted tori by a factor of $\delta$.  Then we cut two slits in each of the two copies of $T$: in the first, we cut two slits, one each the same sizes and angles as the sizes and angles of the slits in $\delta T_{Y_1,-\frac{d}{2}}$ and $\delta T_{Y_2, -\frac{3d}{2}}$ and glue the appropriate pairings along these slits; then we similarly cut two slits in the second torus, one of each size and angle as before, and then attach them to the remaining slits on $\delta T_{Y_1,-\frac{d}{2}}$ and $\delta T_{Y_2, -\frac{3d}{2}}$.  Importantly, we glue them so that the twisting around each of the newly formed curves which bound these subsurfaces (and are the lifts of $\gamma$) is equal to that of the twisting around these curves in $\GG'_1(0)$.  In particular, the twisting around each of the curves lifted from $\gamma$ is coarsely equal.\\

Fix homeomorphisms from $Z_1$ to each of the above double-slitted tori.  This allows us to define a quadratic differential $q \in Q(S)$.  Let $\GG: [0,2d] \rightarrow \TT(S)$ be the Teichm\"uller geodesic segment defined by $q$.  Let $h \in \MCG(S)$ be the involution which rotates $S$ to switch $Y_1$ with $Y_2$ and $Z_1$ with $Z_2$.  We claim the following hold:

\begin{enumerate}
\item $d_T(\GG(0), \GG'_1(0)) \asymp 1$ \label{r:af 1}
\item $d_T(\GG(2d),\GG'_1(2d)) \asymp 1$ \label{r:af 2}
\item $d_T(\GG(d), h\cdot \GG(d)) \asymp d$ \label{r:af 3}
\end{enumerate}

where the constants subsumed by the symbol $\asymp$ depend only on $S$.  We remark that claims (\ref{r:af 1}) and (\ref{r:af 2}) imply that $\GG(0)$ and $\GG(2d)$ are $R$-almost-fixed for some constant $R$, as $\GG'_1(0)$ and $\GG'(2d)$ are fixed.  The content of (\ref{r:af 3}) is that $\GG(d)$ is not $d$-almost-fixed.  The constant $d$ is of our choosing, while $R$ depends only on the topology of $S$.  Thus, verification of (\ref{r:af 1}), (\ref{r:af 2}), and (\ref{r:af 3}) completes the proof of the theorem.\\

The remainder of the proof follows Rafi's closely.  We first show claims (\ref{r:af 1}) and (\ref{r:af 2}) by satisfying the conditions of \cite{Raf14}[Corollary 2.6].  Then we apply \cite{Raf14}[Theorem 4.2] to conclude claim (\ref{r:af 3}) holds.\\

First, note that, by construction, relative twisting around the lifts of $\gamma$ to $S$ with $\GG'_1(0)$ is uniformly bounded.  Second, we note that since the vertical and horizontal foliations of $Y_1,Y_2,Z_1,$ and $Z_2$ are cobounded, no curve in any of them is ever short along $\GG$, so the set of short curves of both $\GG(0)$ and $\GG'_1(0)$ are precisely the lifts of $\gamma$.\\

As for the aforementioned subsurfaces, the restrictions of $q$ to each of  $Y_1,Z_1,$ and $Z_2$ are identical to $q_{Y_1}$, $q_{Z_1}$, and $q_{Z_2}$, which are the projections of $q'_0$ to $Y_1, Z_1,$ and $Z_2$, respectively; similarly, the projection of $q$ to $Y_2$ is identical to $q_{Y_2}$, which is the projection of $\widebar{q}'_0$ to $Y_2$.  By construction, the active intervals along $\GG$ of $Y_1$ and $Y_2$, which we denote $I_{Y_1}, I_{Y_2}$, are $[0,d]$ and $[d,2d]$ respectively.  By Theorem 4.2 of \cite{Raf14}, the projections of $\GG$ to $\TT(Y_1)$ during $I_{Y_1}$ and  to $\TT(Y_2)$ during $I_{Y_2}$ fellow-travel the geodesics defined by the restriction of $q$ to $Y_1$ and $Y_2$, respectively, and outside of these intervals have uniformly bounded projections to $\CC(Y_1)$ and $\CC(Y_2)$.  In particular: 

\begin{enumerate}

\item For any $t \in [0,d]$, we have $d_{\TT(Y_1)}(\GG(t)\big|_{Y_1}, q_{Y_1}) \asymp 1$
\item For $t \in [d,2d]$ we have $d_{\TT(Y_2)}(\GG(t)\big|_{Y_2}, q_{Y_2}) \asymp 1$
\item $d_{Y_1}(\GG(0), \GG'_1(0)) \asymp 1$ and $d_{Y_2}(\GG(2d), \GG'_1(2d)) \asymp 1$ \label{r:af 4}
\end{enumerate}

To finish the proof of claim (\ref{r:af 1}), it remains to show that $\text{Ext}_{\GG(0)}(\gamma') \asymp \text{Ext}_{\GG'_1(0)}(\gamma')$ for each lift $\gamma'$ of $\gamma$.  Of the four lifts of $\gamma$, the two bounding $Y_1$ have coarsely the same extremal length in $\GG(0)$ as they do in $\GG'_1(0)$ for we have scaled them in the same fashion, whereas the two bounding  $Y_2$ have coarsely the same extremal lengths in $\GG(0)$ as they do in $\GG'_2(0)$.  Thus, by the construction of $\GG_1$ and $\GG_2$ in \cite{Raf14}[Theorem 7.3], they have coarsely the same extremal length.\\

It remains to show that claim (\ref{r:af 3}) holds.  Since $Y_1$ is an isolated subsurface along $\GG$ during $[0,d]$ and no curve in $Y_1$ becomes short, it follows from \cite{Raf14}[Theorem 6.1] and \cite{RS09}[Lemma 4.4] that the shadow of $\GG$ in $\CC(Y_1)$ during $[0,d]$ is a parametrized quasigeodesic.  Thus it follows from (\ref{r:af 4}) that
\[d_{\CC(Y_1)}(\GG(0), \GG(d)) \overset{\cdot}{\asymp} d \indent \text{and} \indent d_{\CC(Y_2)}((\GG(0), \GG(d)) \asymp 1\]

Since $Y_1$ and $Y_2$ are homeomorphic, $\CC(Y_1)$ and $\CC(Y_2)$ are isometric.  Let $\Phi: \CC(Y_1) \rightarrow \CC(Y_2)$ be such an identification.  Since $d_{\CC(Y_2)}(\Phi(\GG(d)), \GG(d)) \asymp d$, claim (\ref{r:af 3}) follows from the distance formula Theorem \ref{r:distances}, completing the proof of the theorem.

\end{proof}

\begin{remark}[Generalizations of the counter-example]
We expect that the counter-example constructed in Theorem \ref{r:non qc} should be a common phenomenon.  The construction takes advantage of a surface lifting to disjoint subsurfaces in the covering surface, after which a geodesic is made to move at different times through the subsurfaces.  We expect that nonquasiconvexity should hold any time this phenomenon occurs.  More generally, it would not be surprising if nonquasiconvexity holds any time $\Fix(H)$ has infinite diameter, that is when $\OO$ is not an orbifold with three cone points.
\end{remark}


\begin{thebibliography}{9999}
\bibitem[Ahl61]{Ahl61} L.V. Ahlfors. Some remarks on Teichm\"uller space of Riemann surfaces. Ann. of Math (2), 74: 171-191, 1961.
\bibitem[Ber03]{Ber03} J. Behrstock. Asymptotic Geometry of the Mapping Class Group and Teichmuller Space. Geometry and Topology 10: 1523-1578, 2006.
\bibitem[BG71]{BG71} L. Bers, L. Greenberg. Isomorphisms between Teichm\"uller spaces.  \emph{Advances in Riemann surfaces}, Ann. of Math. Studies, No. 66, p. 53-79. (1971)
\bibitem[BBFS09]{BBFS09} M. Bestvina, K. Bromberg, K. Fujiwara, J. Souto. Shearing coordinates and convexity of length functions on Teichm\"uller space. Amer. J. Math. 135 (2013), no. 6, 1449-1476.
\bibitem[BH99]{BH99} M. Bridson, A. Haefliger. Metric spaces of non-positive curvature, Grundlehren der Mathematischen Wissenschaften, vol. 319, Springer-Verlag 1999.
\bibitem[BM08]{BM08} J. Behrstock, Y. Minsky.  Dimension and rank for mapping class groups.  Annals of Mathematics, 167: 1055-1077, 2008.
\bibitem[Br03]{Br03} J. Brock. The Weil-Petersson metric and volumes of 3-dimensional hyperbolic convex cores. \emph{J. Amer. Math. Soc.}, 16(3): 495-535 (electronic), 2003. arXiv:math.GT/0109048. 
\bibitem[Br05]{Br05} J. Brock. The Weil-Petersson Visual Sphere. Geometriae Dedicata 115(2005),
1-18.
\bibitem[CJ94]{CJ94} A. Casson, and D. Jungreis. Convergence groups and Seifert fibered 3-manifolds, Invent. Math. 118 (1994).
\bibitem[Dur13]{Dur13} M. Durham. Augmented marking complex. arXiv: 1309.4065
\bibitem[EMR13]{EMR13} A. Eskin, H. Masur, K. Rafi. Large scale rank of Teichm\"uller space. arXiv:1307.3733v1
\bibitem[FM02]{FM02} B. Farb, L. Mosher. Convex cocompact subgroups of mapping class groups.  Geometry \& Topology, vol 6 (2002) 91-152
\bibitem[FM12]{FM12} B. Farb, D. Margalit, A primer on mapping class groups, Princeton Mathematical Series, vol. 49, Princeton University Press, Princeton, NJ, 2012.
\bibitem[Gab92]{Gab92} D. Gabai, Convergence groups are Fuchsian groups, Ann. of Math. (2) 136 (1992).
\bibitem[GM08]{GM08} D. Groves and J. F. Manning. Dehn Filling in Relatively Hyperbolic Groups. \emph{Israel Journal of Mathematics}, 168: 317-429, 2008.
\bibitem[HOP12]{HOP12} Realisation and dismantlability.  Geom. Topol. 18 (2014), no. 4, 2079-2126.
\bibitem[Hub]{Hub} J. Hubbard. Teichm\"uller theory and applications to geometry, topology, and dynamics.  Matrix Editions.
\bibitem[Ker83]{Ker83} S. Kerckhoff, The Nielsen realization problem, Ann. of Math. 117 (1983), 235-265.
\bibitem[Kra59]{Kra59} S. Kravetz.  On the geometry of Teichm\"uller spaces and the structure of their modular groups.  Ann. Acad. Sci. Fenn. 278 (1959), 1-35.
\bibitem[LR11]{LR11} A. Lenzhen, K. Rafi.  Length of a curve is quasi-convex along a TeichmŸller geodesic. J. Differential Geom. 88 (2011), no. 2, 267-295
\bibitem[Lin74]{Lin74} M. Linch. Comparison of metrics on Teichm\"uller space. Proceedings of the American Mathematical Society, Vol. 43, No. 2 (1974), 349-352.
\bibitem[Mas75]{Mas75} H. Masur. On a class of geodesics in Teichm\"uller space. Annals of Mathematics, 2nd Ser., Vol. 102, No. 2: 205-221.  (1975) 
\bibitem[Mas76]{Mas76} H. Masur. Extension of the Weil-Petersson metric to the boundary of Teichmuller space. Duke Math. J. 43 (1976), no. 3, 623-635
\bibitem[Mas10]{Mas10} H. Masur. Geometry of Teichm\"uller space with the Teichm\"uller metric. (2010)
\bibitem[Min93]{Min93} Y. N. Minsky. Teichm\"uller geodesics and ends of hyperbolic 3-manifolds. Topology (1993) Vol. 32, No. 3, 625-647.
\bibitem[Min96]{Min96} Y. N. Minsky. Extremal length estimates and product regions in Teichm\"uller space. \emph{Duke Math. J.}, 83(2):249-286, 1996.
\bibitem[Min03]{Min03} Y.N. Minsky. The classification of kleinian surface groups, i: Models and bounds.  Ann. of
Math. (2) 171 (2010), 1-107.
\bibitem[MM99]{MM99} H. Masur and Y. N. Minsky. Geometry of the complex of curves. I. Hyperbolicity. Invent. Math., 138(1):1030149, 1999.
\bibitem[MM00]{MM00} H. Masur and Y. N. Minsky. Geometry of the complex of curves. II. Hierarchical structure. \emph{Geom. Funct. Anal.}, 10(4):902-974, 2000.
\bibitem[MW02]{MW02} H. Masur and M. Wolf. The Weil-Petersson isometry group. Geom. Dedicata, 93:177-190, 2002
\bibitem[Pap07]{Pap07} A. Papadopoulos Ed. Handbook of Teichm\"uller theory I. European Mathematical Society, Zurich, 2007. 
\bibitem[Raf07]{Raf07} K. Rafi. A combinatorial model for the Teichm\"uller metric. Geom. Funct. Anal., 17(3):936-959, 2007.
\bibitem[Raf14]{Raf14} K. Rafi.  Hyperbolicity in Teichm\"uller space. Geom. Topol. 18 (2014), no. 5, 3025-3053.
\bibitem[RS09]{RS09} K. Rafi and S. Schleimer. Covers and the curve complex. Geom. Topol. 13 (2009), no. 4, 2141-2162.
\bibitem[Roy71]{Roy71} H. L. Royden. Automorphisms and isometries of Teichm\"uller space. Ann. Math. Stud. 66(1971), 369-84.
\bibitem[RS09]{RS09} K. Rafi and S. Schleimer. Covers and the curve complex, Preprint, 2007, arXiv:math/0701719v2 [math.GT]
\bibitem[Schleim]{Schleim} S. Schleimer.  Notes on the complex of curves.
\bibitem[Tao13]{Tao13} J. Tao.  Linearly bounded conjugator property for mapping class groups, Geom. Funct. Anal. 23 (2013), no. 1, 415-466.
\bibitem[Tei40]{Tei40} O. Teichm\"uller. Extremale quaskonforme Abbildungen und quadratische Differentiale. Abh. Preuss. Akad. Wiss. Math.-Nat. Kl., 1939(22): 197, 1940.
\bibitem[Tro96]{Tro96} A. Tromba, DirichletÕs energy on TeichmllerÕs moduli space and the Nielsen realization problem, Math. Z. 222 (1996).
\bibitem[Wol86]{Wol86} S. Wolpert. Chern forms of the Riemann tensor for the moduli space of curves. Invent. math. 85, 119-145 (1986)
\bibitem[Wol87]{Wol87} S. Wolpert. Geodesic length functions and the Nielsen problem. J. Diff. Geo. 25 (1987) 275-296.
\bibitem[Wol05]{Wol05} S. Wolpert. The geometry of the Weil-Petersson completion of Teichm\"uller space. ÒSurveys in Differential Geometry VIII (Boston, MA, 2002)Ó, Int. Press, Somerville, MA (2003), 357-393.
\bibitem[Wol07]{Wol07} S. Wolpert.  The Weil-Petersson metric geometry. arXiv:0801.1075v1 (2008).
\end{thebibliography}
\end{document}